\documentclass[11pt,a4paper,english]{article}

\usepackage[latin1]{inputenc}
\usepackage{babel}
\usepackage[centertags]{amsmath}
\usepackage{amsfonts}
\usepackage{amssymb}
\usepackage{color}
\usepackage{amsthm}
\usepackage{newlfont}
\usepackage{graphicx}
\usepackage{dsfont}
\usepackage{geometry}
\usepackage{mathrsfs}
\usepackage{authblk}
\usepackage{verbatim}
\usepackage{hyperref}
\usepackage[normalem]{ulem}

\newtheorem{theorem}{Theorem}[section]
\newtheorem{lemma}[theorem]{Lemma}

\newtheorem{prop}[theorem]{Proposition}

\newtheorem{remark}[theorem]{Remark}
\newtheorem{ass}[theorem]{Assumption}

\def\RR{\mathbb R}
\def\EE{\mathsf E}
\def\PP{\mathsf P}
\def\CC{\mathcal C}
\def\DD{\mathcal D}

\def\cF{{\cal F}}

\def\cS{{\cal S}}

\def\bl{\textcolor{blue}}

\newcommand{\LL}{\mathbb{L}_X}
\newcommand{\R}{\mathbb{R}}

\newcommand{\beq}{\begin{equation}}
\newcommand{\eeq}{\end{equation}}
\newcommand{\bea}{\begin{eqnarray}}
\newcommand{\eea}{\end{eqnarray}}
\newcommand{\beas}{\begin{eqnarray*}}
\newcommand{\eeas}{\end{eqnarray*}}

\makeatletter  
\@addtoreset{equation}{section}
\def\theequation{\arabic{section}.\arabic{equation}}
\makeatother  
\begin{document}

\title{{A Non Convex Singular Stochastic Control Problem \\ and its Related Optimal Stopping Boundaries}\footnote{The first and the third authors were supported by EPSRC grant EP/K00557X/1; financial support by the German Research Foundation (DFG) via grant Ri--1128--4--1 is gratefully acknowledged by the second author.}}

\author{Tiziano De Angelis\thanks{School of Mathematics, The University of Manchester, Oxford Road, Manchester M13 9PL, United Kingdom; \texttt{tiziano.deangelis@manchester.ac.uk}}\:\:\:\:Giorgio Ferrari\thanks{Center for Mathematical Economics, Bielefeld University, Universit\"atsstrasse 25, D-33615 Bielefeld, Germany; \texttt{giorgio.ferrari@uni-bielefeld.de}}\:\:\:\:John Moriarty\thanks{School of Mathematics, The University of Manchester, Oxford Road, Manchester M13 9PL, United Kingdom; \texttt{john.moriarty@manchester.ac.uk}}}

\date{\today}
\maketitle

\textbf{Abstract.}
Equivalences are known between problems of singular stochastic control (SSC) with convex performance criteria and related questions of optimal stopping, see for example Karatzas and Shreve [SIAM J.~Control Optim.~$22$ (1984)].  
The aim of this paper is to investigate how far connections of this type generalise to a non convex problem of purchasing electricity. Where the classical equivalence breaks down we provide alternative connections to optimal stopping problems.

We consider a non convex infinite time horizon SSC problem whose state consists of an uncontrolled diffusion representing a real-valued commodity price, and a controlled increasing bounded process representing an inventory.  We analyse the geometry of the \emph{action} and \emph{inaction} regions by characterising their (optimal) boundaries. Unlike the case of convex SSC problems we find that the optimal boundaries may be both \emph{reflecting} and \emph{repelling} and it is natural to interpret the problem as one of SSC with discretionary stopping.

\smallskip


\smallskip

{\textbf{MSC2010 subsject classification}}: 91B70, 93E20, 60G40, 49L20.

\smallskip



\section{Introduction and Problem Formulation}
\label{Introduction}

It is well known that convexity of the performance criterion suffices to link certain singular stochastic control problems to related problems of optimal stopping (cf.\ \cite{ElKK88}, \cite{KaratzasShreve84} and \cite{K85}, among others). In this paper we establish multiple connections with optimal stopping for a non convex, infinite time-horizon, two-dimensional, degenerate singular stochastic control problem motivated by a problem of purchasing electricity. The non convexity arises because our electricity price model allows for both positive and negative prices.

We model the purchase of electricity over time at a stochastic real-valued spot price $(X_t)_{t \geq 0}$ for the purpose of storage in a battery (for example, the battery of an electric vehicle). The battery must be full at a random terminal time, any deficit being met by a less efficient charging method. This feature is captured by inclusion of a terminal cost term equal to the product of the terminal spot price and a convex function $\Phi$ of the undersupply. Under the assumption of a random terminal time independent of $X$ and exponentially distributed, we show in Appendix \ref{app:formulation} that this optimisation problem is equivalent to solving the following problem.

Letting $\lambda>0$ and $c \in [0,1]$ be constants, $\{\nu: \nu \in \cS_c\}$ a set of bounded increasing controls, $(X^x_t)_{t\ge0}$ a continuous strong Markov process starting from $x\in\RR$ at time zero and $C^{c,\nu}_t$ a process representing the level of storage at time $t$:
\beq
\label{ControlledY}
C^{c,\nu}_t= c + \nu_t, \quad t \geq 0,
\eeq
the problem is to find
\beq
\label{valuefunction}
U(x,c):=\inf_{\nu \in \mathcal{S}_c} \mathcal{J}_{x,c}(\nu),
\eeq
with
\beq
\label{nonconvex}
\mathcal{J}_{x,c}(\nu):=\EE\bigg[\int_0^{\infty}e^{-\lambda s}\lambda X^x_{s}\Phi(
C^{c,\nu}_{s})ds + \int^{\infty}_0 {e^{-\lambda s}X^x_s\,d{\nu}_s} \bigg],
\eeq
and the minimising control policy $\nu^*$.
It is notable that the integrands in \eqref{nonconvex} may assume both positive and negative values: economically this corresponds to the possibility that the price is negative prior to or at the random terminal time in the original optimisation problem discussed in Appendix \ref{app:formulation}.

In common with other commodity prices, 
the standard approach in the literature is to model electricity prices through a geometric or arithmetic mean reverting process (see, e.g., \cite{GR} or \cite{LS} and references therein). Motivated by deregulated electricity markets with renewable generation, in which periods of negative electricity prices have been observed due to the requirement to balance real-time supply and demand, we assume an arithmetic model. We assume that $X$ follows a standard time-homogeneous Ornstein-Uhlenbeck process\footnote{See Appendix \ref{factsOU} for general facts on the Ornstein-Uhlenbeck process.} with positive volatility $\sigma$, positive adjustment rate $\theta$ and positive asymptotic (or equilibrium) value $\mu$. On a complete probability space $(\Omega,\cF,\PP)$, with $\mathbb{F}:=(\cF_t)_{t\ge0}$ the filtration generated by a one-dimensional standard Brownian motion $(B_t)_{t\ge0}$ and augmented by $\PP$-null sets, we therefore take $X^x$ as the unique strong solution of
\beq
\label{OU}
\left\{
\begin{array}{ll}
dX^x_t= \theta(\mu-X_t^x)dt + \sigma dB_t, \quad t>0,\\
X^x_0=x\in\RR.
\end{array}
\right.
\eeq

We assume that the electricity storage capacity is bounded above by 1 (this resembles a so-called {\em finite-fuel} constraint, see for example \cite{ElKK88}): for any initial level $c \in [0,1]$ the set of admissible controls is
\begin{eqnarray}
\label{admissiblecontrols}
\mathcal{S}_c \hspace{-0.2cm}&:=&\hspace{-0.2cm} \{\nu:\Omega \times \mathbb{R}_{+} \mapsto  \mathbb{R}_{+}, (\nu_{t}(\omega))_{t \geq 0}\mbox{ is nondecreasing,\,\,left-continuous, adapted} \\
&& \hspace{5cm} \mbox{ with} \,\,c + \nu_t\leq 1\,\,\,\forall t \geq 0, \nonumber \,\,\nu_0=0\,\,\,\,\PP-\mbox{a.s.}\},
\end{eqnarray}
and $\nu_t$ represents the cumulative amount of energy purchased up to time $t$.
From now on we make the following standing assumption on the running cost function $\Phi$.
\begin{ass}
\label{ass-Phi}
$\Phi: \mathbb{R} \mapsto \mathbb{R}_+$ lies in $C^2(\mathbb{R})$ and is decreasing and strictly convex with $\Phi(1)=0$.
\end{ass}
\noindent We note that we do not cover with Assumption \ref{ass-Phi} the case of a linear running cost function, although the solution in the linear case is simpler and follows immediately from the results contained in Sections \ref{Case2} and \ref{Case3} below.

With these specifications problem \eqref{valuefunction} shares common features with the class of finite-fuel, singular stochastic control problems of monotone follower type (see, e.g., \cite{BSW80}, \cite{CMR85}, \cite{ElKK88}, \cite{ElKK91}, \cite{K85} and \cite{KaratzasShreve86} as classical references on finite-fuel monotone follower problems). 
Such problems, with finite or infinite fuel and a running cost (profit) which is convex (concave) in the control variable, have been well studied for over 30 years (see, e.g., \cite{Alvarez99}, \cite{Alvarez01}, \cite{Bank}, \cite{CH94}, \cite{ElKK88}, \cite{ElKK91}, \cite{K83}, \cite{KaratzasShreve84}, \cite{K85} and \cite{KaratzasShreve86}, among many others). Remarkably it turns out that convexity (or concavity), together with other more technical conditions, is sufficient to prove that such singular stochastic control problems are equivalent to related problems of optimal stopping; moreover the optimally controlled state process is the solution of a Skorokhod reflection problem at the free-boundary of the latter (see, e.g., \cite{CH94}, \cite{ElKK88}, \cite{KaratzasShreve84}, \cite{K85} and \cite{KaratzasShreve86}).

In our case the weighting function $\Phi$ appearing in the running cost is strictly convex, the marginal cost $e^{-\lambda s}X^x_s\,d{\nu}_s$ of exercising control is linear in the control variable, and the set of admissible controls $\mathcal{S}_c$ (cf.\ \eqref{admissiblecontrols}) is convex. However the Ornstein-Uhlenbeck process $X^x$ of \eqref{OU} can assume negative values with positive probability and is also a factor of the running cost so that the total expected cost functional \eqref{nonconvex} is not convex in the control variable. Therefore the connection between singular stochastic control and optimal stopping as addressed in \cite{ElKK88}, \cite{KaratzasShreve84} and \cite{K85}, among others, is no longer guaranteed for problem (\ref{valuefunction}).

The optimisation problem we study (in common with many others in the literature, see for instance \cite{FedericoPham}, \cite{Ferrari2012}, \cite{Kobila}, \cite{MehriZervos} or \cite{RS}) has two state variables, one which is diffusive and the other which is a control process, a setup typically referred to as degenerate two-dimensional. Our particular problem may be regarded as a two-dimensional (history dependent) relative of a class of one-dimensional problems studied for example in a series of papers by Alvarez (see \cite{Alvarez98}, \cite{Alvarez99} and \cite{Alvarez01} and references therein). The latter problems are neither convex nor concave, and the `critical depensation' which they exhibit is also observed in the solutions we find. Their solutions are, however, found in terms of optimal boundaries represented by points on the real axis rather than the free boundary curves studied in the present paper. An advantage of the one-dimensional setting is that general theory may be applied to develop solutions for general diffusion processes. Since additional arguments are required to verify the optimality of the free boundaries in our two-dimensional degenerate setting, however, such generality does not seem achievable and we work with the specific class of Ornstein-Uhlenbeck processes given by \eqref{OU}.

We now briefly summarise the main findings that will be discussed and proved in detail in Sections \ref{sec:connection}, \ref{Case2}, \ref{Case3} and \ref{CompleteCase}. We begin in Section \ref{sec:connection} with a useful restatement of the problem \eqref{valuefunction} as a singular stochastic control problem with discretionary stopping (SSCDS) (see Eq.\ \eqref{SSCDS} below). To the best of our knowledge SSCDS problems were originally introduced in \cite{DZ}. In that paper the authors aimed at minimising total expected costs with a quadratic running cost depending on a Brownian motion linearly controlled by a bounded variation process, and with a constant cost of exercising control. The case of finite-fuel SSCDS was then considered in \cite{KOWZ} were a terminal quadratic cost at the time of discretionary stopping was also included.
A detailed analysis of the variational inequalities arising in singular control problems with discretionary stopping may be found in \cite{Morimoto2003} and \cite{Morimoto2010}.

Our SSCDS problem \eqref{valuefunction} exhibits three regimes depending on the sign of the function
\begin{equation}
\label{def-k}
k(c):=\lambda+\theta+\lambda \,\Phi'(c)
\end{equation}
over $c \in [0,1]$. 
We will show (Section \ref{Case3}) that for fixed $c$, the sign of the function $k$ determines the nature of the relationship between the price level $x$ and the net contribution to the infimum \eqref{valuefunction} (equivalently, the infimum \eqref{SSCDS}) from exercising control. In particular, when $k>0$ this relationship is increasing and when $k<0$ it is decreasing.

Since $c \mapsto k(c)$ is strictly increasing by the strict convexity of $\Phi$ (cf.\ Assumption \ref{ass-Phi}) define $\hat{c} \in \mathbb{R}$ as the unique solution of
\beq
\label{c_one}
k(c) = 0
\eeq
should one exist, in which case $\hat{c}$ may belong to $[0,1]$ or not depending on the choice of $\Phi$ and on the value of the parameters of the model.

In Section \ref{Case2} we study the case in which $k(c) \geq 0$ for all $c \in [0,1]$ (and hence $\hat{c} \leq 0$, if it exists). We show that although problem \eqref{valuefunction} is non convex, the optimal control policy behaves as that of a convex finite-fuel singular stochastic control problem of monotone follower type (cf., e.g., \cite{ElKK88}, \cite{K85} and \cite{KaratzasShreve86}) and, accordingly, (i) the optimal control $\nu^*$ is of the reflecting type, being the minimal effort to keep the (optimally) controlled state variable inside the closure of the continuation region of an associated optimal stopping problem up to the time at which all the fuel has been spent, and (ii) the directional derivative $U_c$ of \eqref{valuefunction} in the $c$ variable coincides with the value function of the associated optimal stopping problem. In this case the infimum over stopping times is not achieved in the SSCDS formulation \eqref{SSCDS}, which may be interpreted as a formally infinite optimal stopping time.

On the other hand, in Section \ref{Case3} we assume $k(c)\leq 0$ for all $c \in [0,1]$ (and hence $\hat{c} \geq 1$). In this case the optimal singular control policy in \eqref{SSCDS} is identically zero, which may be interpreted as problem \eqref{valuefunction} becoming a stopping problem in which it is optimal to do nothing up to the first hitting time of $X$ at a repelling barrier (in the language of \cite{KOWZ}) and then to exercise all the available control. In particular the differential connection between SSC and optimal stopping observed in the previous case breaks down here and, to the best of our knowledge, this is a rare example of such an effect in the literature on SSC problems.

The case when $\hat{c}$ exists in $[0,1]$ is discussed in Section \ref{CompleteCase}. This case, in general involving multiple free-boundaries, is left as an open problem although we refer to a complete solution of the limiting case $\theta=0$ (cf.\ \eqref{OU}) derived in a companion paper \cite{DeAFeMo14b}. Finally, we collect in the Appendix the model formulation, some well known facts on the Ornstein-Uhlenbeck process $X$ and some technical results.

Before concluding this section we observe that problem \eqref{valuefunction} may also fit in the economic literature as an irreversible investment problem with stochastic investment cost. It is well known that in the presence of a convex cost criterion (or concave profit) the optimal (stochastic) irreversible investment policy consists in keeping the production capacity at or above a certain reference level $\ell$ (see, e.g., \cite{Bertola}, \cite{DP} and \cite{Pindyck}; cf.~also \cite{BK97} among others for the case of stochastic investments cost) which has been recently characterized in \cite{Ferrari2012} and \cite{RS} where it is referred to as {\em base capacity}. The index $\ell_t$ describes the desirable level of capacity at time $t$. If the firm has capacity $C_t > \ell_t$, then it faces excess capacity and should wait. If the capacity is below $\ell_t$, then it should invest $\nu_t = \ell_t - C_t$ in order to reach the level $\ell_t$.

Our analysis shows that in presence of non-convex costs it is not always optimal to invest just enough to keep the capacity at or above a base capacity level. In fact, for a suitable choice of the parameters ($\hat{c} \leq 0$) the optimal investment policy is of a purely dichotomous {\em bang-bang} type: not invest or go to full capacity. On the other hand, for a different choice of the parameters ($\hat{c} \geq 1$) a base capacity policy is optimal regardless of the non convexity of the total expected costs. To the best of our knowledge this result is a novelty also in the mathematical-economic literature on irreversible investment under uncertainty.


\subsection{A Problem with Discretionary Stopping}
\label{sec:connection}

In this section we establish the equivalence between problem \eqref{valuefunction} and a finite-fuel singular stochastic control problem with discretionary stopping (cf.~\cite{DZ} and \cite{KOWZ} as classical references on this topic). We first observe that, for fixed $(x,c)\in\RR\times[0,1]$ and any $\nu\in\cS_c$, the process $(X^x_t)_{t\ge0}$ and the processes $(I^{x,c,\nu}_t)_{t\ge 0}$, $(J^{x,\nu}_t)_{t\ge 0}$ defined by
\begin{align}\label{eq:IJ}
&I^{x,c,\nu}_t:=\int_0^{t}e^{-\lambda s}
\lambda X^x_s \Phi(C^{c,\nu}_s)ds\qquad\text{and}\qquad J^{x,\nu}_t:=\int^t_0{e^{-\lambda s}X^x_sd\nu_s},
\end{align}
respectively, are uniformly bounded in $L^2(\Omega,\PP)$, hence uniformly integrable. This is a straightforward consequence of standard properties of the Ornstein-Uhlenbeck process \eqref{OU} (see Appendix \ref{factsOU}), Assumption \ref{ass-Phi}, the finite fuel condition and an integration by parts.

\begin{prop}
\label{discretionarystopping}
Recall $U$ of \eqref{valuefunction}. Then one has $U \equiv \hat{U}$ with
\beq
\label{SSCDS}
\hat{U}(x,c)=\inf_{\nu \in \mathcal{S}_c,\, \tau \geq 0}\EE\bigg[\int_0^{\tau}e^{-\lambda s}\lambda X^x_{s}\Phi(
C^{c,\nu}_{s})ds + \int^{\tau}_0 {e^{-\lambda s}X^x_s\,d{\nu}_s} + e^{-\lambda \tau}X^x_{\tau}(1-C^{c,\nu}_{\tau})\bigg]
\eeq
for $(x,c)\in\RR\times[0,1]$ and where $\tau$ must be a $\PP$-a.s.\ finite stopping time.
\end{prop}
\begin{proof}
Fix $(x,c)\in\RR\times[0,1]$. Take a sequence of deterministic stopping times $(t_n)_{n \in \mathbb{N}}$ such that $t_n \uparrow \infty$ as $n \uparrow \infty$ in the expectation in \eqref{SSCDS} and use uniform integrability, continuity of $X^x_{\cdot}$, $I^{x,c,\nu}_{\cdot}$, left-continuity of $J^{x,\nu}_{\cdot}$ (cf.~\eqref{eq:IJ}) and that $\lim_{n \uparrow \infty} \EE[e^{-\lambda t_n}X^x_{t_n}(1 - C^{c,\nu}_{t_n})]=0$, to obtain $\hat{U}\leq U$ in the limit as $n\to\infty$. To show the reverse inequality, for any admissible $\nu \in \mathcal{S}_c$ and any stopping time $\tau \geq 0$ set
\beq
\label{nutau}
\hat{\nu}_t:=\left\{
\begin{array}{ll}
\nu_t, & t \leq \tau, \\
\\
1-c, & t > \tau.
\end{array}
\right.
\eeq
The control $\hat{\nu}$ is admissible and then from the definition of $U$ (cf.\ \eqref{valuefunction}) it follows that
$$U(x,c) \leq \mathcal{J}_{x,c}(\hat{\nu}) = \EE\bigg[\int_0^{\tau}e^{-\lambda s}\lambda X^x_{s}\Phi(
C^{c,\nu}_{s})ds + \int^{\tau}_0 {e^{-\lambda s}X^x_s\,d{\nu}_s} + e^{-\lambda \tau}X^x_{\tau}(1-C^{c,\nu}_{\tau})\bigg].$$
Since the previous inequality holds for any admissible $\nu$ and any $\PP$-a.s.\ finite stopping time $\tau \geq 0$ we conclude that $U \leq \hat{U}$, hence $U \equiv \hat{U}$.
\end{proof}

Since the proof of Proposition \ref{discretionarystopping} does not rely on particular cost functions (running cost and cost of investment), the arguments apply to a more general class of SSC problems. However in some cases (including the convex or concave SSC problems) it turns out that {the infimum over stopping times in \eqref{SSCDS} is not achieved} 
and one should formally take $\tau = +\infty$: clearly in those cases an equivalence such as Proposition \ref{discretionarystopping} would add no insight to the analysis of the problem.
In contrast we show below that depending on the quantity $\hat{c}$ introduced through \eqref{c_one}, both the control and stopping policies in \eqref{SSCDS} may play either trivial or nontrivial roles through the interplay of two free-boundaries. A complete analysis of the interplay of these two free-boundaries is outside the scope of this paper and a challenging open problem (discussed in Section \ref{CompleteCase}).


\section{The Case $\hat{c} \; \leq \; 0$}
\label{Case2}

In this section we identify when the differential relationship between SSC and optimal stopping known in convex problems of monotone follower type with finite fuel (cf., e.g., \cite{ElKK88}, \cite{K85} and \cite{KaratzasShreve86}) holds in our non-convex problem. In this case, as discussed above one should formally set $\tau^*=+\infty$ in \eqref{SSCDS}.
We find that the differential relationship holds when $k(c)>0$ (cf.\ \eqref{def-k}) for all $c \in [0,1]$ or, equivalently, when $\hat{c} <0$: in this case the derivative (with respect to $c$, the direction of the control variable) of the value function in  \eqref{valuefunction} is given by the value function of the family of optimal stopping problems solved below and the optimal control $\nu^*$ is of reflection type, being the minimal effort to keep the (optimally) controlled state variable $C^{c,\nu^*}$ above the corresponding non-constant free boundary. The case $\hat c = 0$ is similar, see Remark \ref{limitbstar}.





\subsection{The Associated Family of Optimal Stopping Problems}\label{aosp}

The family of infinite time-horizon optimal stopping problems we expect to be naturally associated to the control problem \eqref{valuefunction} is given by
\begin{align}
\label{opt-st}
v(x;c):=\sup_{\sigma\ge0}\EE\bigg[-e^{-\lambda\sigma} X^x_\sigma+\int_0^{\sigma}e^{-\lambda s}\lambda\,X^x_{s}\Phi'(c)ds\bigg], \quad c \in [0,1],
\end{align}
where the supremum is taken over all $\PP$-a.s.\ finite stopping times $\sigma$ (see, for example, \cite{ElKK88}, \cite{KaratzasShreve84} or \cite{K85}, among others). For any given value of $c \in [0,1]$, \eqref{opt-st} is a one-dimensional optimal stopping problem that can be addressed through a variety of well established methods. As $c$ varies the optimal stopping boundary points for problem \eqref{opt-st} will serve to construct the candidate optimal boundary of the action region of problem \eqref{valuefunction} and, as noted in the Introduction, we will therefore require sufficient monotonicity and regularity of this free boundary curve to verify its optimality. 

Define
\begin{eqnarray}
\displaystyle G(x;c) := \frac{\mu(k(c)-\theta)}{\lambda} + \frac{k(c)(x-\mu)}{\lambda+\theta},\qquad(x,c)\in\RR\times[0,1], \label{Gexpression}
\end{eqnarray}
\begin{align}\label{x0-01}
x_0(c):=-\,\frac{\theta\mu\Phi'(c)}{k(c)}>0,\qquad c\in[0,1]
\end{align}
and let $\LL$ be the infinitesimal generator of the diffusion $X^x$, i.e.
\begin{align}\label{def:LX}
\mathbb{L}_{X}f\,(x):=\frac{1}{2}\sigma^2 f''(x) + \theta(\mu - x)f'(x),\quad\text{for $f\in C^2_b(\RR)$ and $x\in\RR$.}
\end{align}
The next theorem is proved in Appendix \ref{app-proof1} and provides a characterisation of $v$ in \eqref{opt-st} and of the related optimal stopping boundary.

\begin{theorem}
\label{verifthm01}
For each given $c\in[0,1]$ one has $v(x;c)=-x+u(x;c)$ where
\begin{align}
\label{def-u00}
u(x;c):=\left\{
\begin{array}{ll}
G(x;c)-\frac{G(\beta_*(c);c)}{\phi_\lambda(\beta_*(c))}\phi_\lambda(x), & x>\beta_*(c) \\[+4pt]
0, & x\le\beta_*(c)
\end{array}
\right.
\end{align}
with $\phi_\lambda$ the strictly decreasing fundamental solution of $\LL f=\lambda f$ (cf.~\eqref{phi} in Appendix) and ${\beta_*(c)}\in(-\infty, x_0(c))$ the unique solution of problem:
\begin{align}
\label{smfit01}
\text{find $x\in\RR$:}\quad G_x(x;c)-\frac{G(x;c)}{\phi_{\lambda}(x)}\phi'_{\lambda}(x)=0.
\end{align}
Moreover
\begin{align}\label{eq:sigmastar}
\sigma^*:=\inf\{t \geq 0:X^x_t \leq {\beta_*(c)}\}
\end{align}
is an optimal stopping time in \eqref{opt-st} and $c\mapsto \beta_*(c)$ is strictly decreasing and, if $\hat{c} < 1$, it is $C^1$ on $[0,1]$.
\end{theorem}

\begin{remark}\label{rem:gendif}
The monotonicity of the boundary, crucial for the verification theorem below, is obtained using specific properties of the diffusion $X$ (through the function $\phi_\lambda$) and of the cost functional. To the best of our knowledge general results of this kind for a wider class of diffusions cannot be provided in this non-convex setting either by probabilistic or analytical methods; thus a study on a case by case basis is required.
We note in fact that in \cite{DeAFeMo14b} in a setting similar to the present one but with a different choice of $X$ the geometry of the action and inaction regions for the control problem is quite different.
\end{remark}

\begin{remark}
\label{limitbstar}
In the case when $\hat{c}=0$ (cf.~\eqref{c_one}) one only has $\beta_* \in C^1((0,1])$, as in fact $\lim_{c\downarrow \hat{c}}{\beta_*(c)}=+\infty$ along with its derivative. This follows by noting that taking $y=\beta_*(c)$ in \eqref{smfit01} and passing to the limit as $c\downarrow\hat{c}$, if $\lim_{c\downarrow\hat{c}}\beta_*(c)=\ell<+\infty$ one finds a contradiction. For $c=\hat{c}$ the optimal stopping time for problem \eqref{opt-st} is $\sigma^*=0$ for any $x \in \mathbb{R}$.
\end{remark}

\subsection{The Solution of the Stochastic Control Problem}\label{sscsolution}

In this section we aim at providing a solution to the finite-fuel singular stochastic control problem \eqref{valuefunction} by starting from the solution of the optimal stopping problem \eqref{opt-st} (see also \eqref{osprob}) and guessing that the classical connection to singular stochastic control holds.

By Theorem \ref{verifthm01} we know that $c \mapsto {\beta_*(c)}$ is strictly decreasing and so has a strictly decreasing inverse. We define
\begin{align}\label{def-gstar}
g_*(x):=
\left\{
\begin{array}{ll}
1, & x\le \beta_*(1)\\[+2pt]
 \beta^{-1}_*(x), & x\in(\beta_*(1),\beta_*(0))\\[+2pt]
0, & x\ge \beta_*(0).
\end{array}
\right.
\end{align}
Obviously $g_*: \mathbb{R} \to [0,1]$ is continuous and decreasing. Moreover, since $\beta_* \in C^1$ and $\beta_*' < 0$ (cf.\ again Theorem \ref{verifthm01}), then $g_*'$ exists almost everywhere and it is bounded.

Define the function
\beq
\label{def-F}
F(x,c) := -\int_c^1v(x;y)dy=x(1-c)-\int_c^1u(x;y)dy.
\eeq
We expect that $F(x,c)=U(x,c)$ for all $(x,c) \in \mathbb{R}\times [0,1]$, with $U$ as defined in \eqref{valuefunction}.
\begin{prop}
\label{listprops}
The function $F(x,c)$ in \eqref{def-F} is such that $x\mapsto F(x,c)$ is concave, $F\in C^{2,1}(\mathbb{R}\times[0,1])$ and the following bounds hold
\begin{align}\label{boundsF}
\big|F(x,c)\big| + \big|F_c(x,c)\big|\le C_1(1+|x|),\quad \big|F_x(x,c)\big| + \big|F_{xx}(x,c)\big|\le C_2
\end{align}
for $(x,c)\in\RR\times[0,1]$ and some positive constants $C_1$ and $C_2$.
\end{prop}
\begin{proof}
In this proof we will often refer to the proof of Theorem \ref{verifthm01} in Appendix \ref{app-proof1}. Recall \eqref{def-u00} and that $u^{\beta_*} \equiv u$ (cf.\ Theorem \ref{verifthm01}). Concavity of $F$ as in \eqref{def-F} easily follows by observing that $x\mapsto u(x;c)$ is convex (cf.\ again Theorem \ref{verifthm01}). It is also easy to verify from \eqref{Gexpression} and \eqref{def-u00} that $u$ is of the form $u(x;c)=A(c)P(x)+B(c)$ for suitable continuous functions $A$, $B$ and $P$, so that $(x,c)\mapsto F(x,c)$ is continuous on $\mathbb{R}\times[0,1]$ and $c \mapsto F_c(x,c)$ is continuous on $[0,1]$ as well. From the definition of $u^{\beta_*}$ (cf.~\eqref{def-u00}), \eqref{smfit01}, convexity of $u^{\beta_*}$ and continuity of $\beta_*$ it is straightforward to verify that for $x\in K\subset\mathbb{R}$, $K$ bounded, $|u_x|$ and $|u_{xx}|$ are at least bounded by a function $Q_K(c)\in L^1(0,1)$. It follows that evaluating $F_x$ and $F_{xx}$ one can pass derivatives inside the integral in \eqref{def-F} so to obtain
\begin{align}
\label{Fx}
F_x(x,c) = (1-c) - \int_c^1 u_x(x;y)dy=(1-c) - \int_{g_*(x)\vee c}^1 u_x(x;y)dy
\end{align}
and
\begin{align}
\label{Fxx}
F_{xx}(x,c) = - \int_c^1 u_{xx}(x;y)dy= - \int_{g_*(x)\vee c}^1 u_{xx}(x;y)dy.
\end{align}
Therefore $F\in C^{2,1}$ by \eqref{def-u00}, \eqref{smfit01}, convexity of $u$ (cf.\ Theorem \ref{verifthm01}) and continuity of $g_*(\cdot)$ (cf.\ \eqref{def-gstar}).

Recall now that $\phi_\lambda(x)$ and all its derivatives go to zero as $x\to\infty$ and \eqref{def-gstar}. Then bounds \eqref{boundsF} follow from \eqref{def-u00}, \eqref{def-F}, \eqref{Fx} and \eqref{Fxx}.
\end{proof}

From standard theory of stochastic control (e.g., see \cite{FlemingSoner}, Chapter VIII), we expect that the value function $U$ of \eqref{valuefunction} identifies with an appropriate solution $w$
to the Hamilton-Jacobi-Bellman (HJB) equation
\beq
\label{HJB-U}
\max\{-\mathbb{L}_X w+ \lambda w-\lambda x \Phi(c),-w_c-x\}=0\quad\text{for a.e.~$(x,c)\in\RR\times[0,1]$}.
\eeq
Recall Proposition \ref{listprops}.
\begin{prop}
\label{prop6}
For all $(x,c) \in \mathbb{R} \times [0,1]$ we have that F is a classical solution of \eqref{HJB-U}.
\end{prop}
\begin{proof}
First we observe that \eqref{Gexpression} and \eqref{def-F} give
\begin{align} \label{str}
F(x,c)\hspace{-1pt}=\hspace{-1pt}\mu\Phi(c)\hspace{-1pt}+\hspace{-1pt}(x-\mu)\frac{\lambda\Phi(c)}{\lambda+\theta}
+\phi_{\lambda}(x)\hspace{-3pt}\int_c^1
\hspace{-2pt}\frac{G({\beta_*(y)};y)}{\phi_{\lambda}({\beta_*(y)})}dy\quad\text{for all $c>g_*(x)$}
\end{align}
For any fixed $c\in[0,1]$ and $x\in\RR$ such that $F_c(x,c)>-x$, i.e.~$c>g_*(x)$ (cf.~\eqref{def-F}), one has
\begin{align*}
(\mathbb{L}_X- \lambda)F(x,c)=-\lambda\Phi(c)x
\end{align*}
by \eqref{str}. On the other hand, for arbitrary $(x,c)\in\mathbb{R}\times[0,1]$ we notice that
\begin{equation*}
\label{eq:ffuu}
(\mathbb{L}_X-\lambda )F(x,c)=(1-c)(\theta\mu-(\lambda+\theta)x)-\int_c^1(\mathbb{L}_X-\lambda)u(x;y)dy
\end{equation*}
by \eqref{Fx} and \eqref{Fxx}. Now, recalling \eqref{eq:diffu} one has
$$
\int_c^1(\mathbb{L}_X- \lambda)u(x;y) dy \leq \int_c^1 [\theta\mu-k(y)x] dy = [\theta\mu-(\lambda+\theta)x](1-c)+\lambda\Phi(c)x,
$$
since $\theta\mu-k(c)x \geq 0$ when $F_c(x,c)=-x$, i.e.~$c<{g_*(x)}$, by \eqref{eq:xhat}. Then
\begin{align*}
(\mathbb{L}_X-\lambda)F(x,c) \geq -\lambda \Phi(c)x\quad\text{for all $(x,c) \in \mathbb{R} \times [0,1]$.}
\end{align*}
\end{proof}

We now aim at providing a candidate optimal control policy $\nu^*$ for problem \eqref{valuefunction}. Let $(x,c) \in \mathbb{R} \times [0,1]$ and consider the process
\beq
\label{candidateoptimalcontrol}
\nu_t^*=\Big[g_*\big(\inf_{0 \leq s \leq t} X^x_s\big)-c\Big]^+, \qquad t>0, \quad
\nu_0^*= 0,
\eeq
with $g_*$ as in \eqref{def-gstar} and $[\, \cdot \,]^+$ denoting the positive part.
\begin{prop}
\label{admissiblenustar}
The process $\nu^*$ of \eqref{candidateoptimalcontrol} is an admissbile control.
\end{prop}
\begin{proof}
Fix $\omega\in\Omega$ and recall \eqref{admissiblecontrols}. By definition $t\mapsto \nu^*_t(\omega)$ is clearly increasing and such that $C^{c,\nu^*}_t(\omega) \leq 1$, for any $t\geq 0$, since $0\le g_*(x)\le 1$, $x\in\RR$. The map $x \mapsto g_*(x)$ is continuous, then $t \mapsto \nu^*_t(\omega)$ is continuous, apart of a possible initial jump at $t=0$, by continuity of paths $t\mapsto X^x_t(\omega)$.

To prove that $\nu^* \in \mathcal{S}_c$ it thus remains to show that $\nu^*$ is $(\mathcal{F}_t)$-adapted. To this end, first of all notice that continuity of $g_*(\cdot)$ also implies its Borel measurability and hence progressive measurability of the process $g_*(X^x)$. Then $\nu^*$ is progressively measurable since $g_*\big(\inf_{0 \leq s \leq t} X^x_s\big) = \sup_{0 \leq s \leq t} g_*(X^x_s)$, by monotonicity of $g_*$, and by \cite{DM}, Theorem IV.33. Hence $\nu^*$ is $(\mathcal{F}_t)$-adapted.
\end{proof}
To show optimality of $\nu^*$ we introduce the \emph{action} and \emph{inaction} sets
\begin{align}
\label{def:CSsets}
\CC:=\big\{(x,c)\,:\,F_c(x,c)>-x\big\}\quad\text{and}\quad
\DD:=\big\{(x,c)\,:\,F_c(x,c)=-x\big\},
\end{align}
respectively and with $(x,c)\in\RR\times[0,1]$. Their link to the sets defined in \eqref{regions} is clear by recalling that $F_c=u$. The following Proposition, which is somewhat standard (see, e.g., \cite{KS}, p.\ 210 and \cite{Skorokhod} as classical references on the topic), is proved in Appendix \ref{app-proof2}.
\begin{prop}
\label{C-opt}
Let $C^*_t:=C^{c,\nu^*}_t=c+\nu^*_t$, with $\nu^*$ as in \eqref{candidateoptimalcontrol}. Then $\nu^*$ solves the Skorokhod problem
\begin{enumerate}
\item $\displaystyle (C^*_t, X^x_t) \in \mathcal{\overline{C}}$, $\PP$-almost surely, for each $t > 0$;
\item $\displaystyle \int_0^T e^{-\lambda t}\mathds{1}_{\{(C^*_t, X^x_t) \in \mathcal{C}\}}d\nu_t^*=0$ almost surely, for all $T \geq 0$,
\end{enumerate}
where $\mathcal{\overline{C}}:=\{(x,c)\in\RR\times[0,1]:c \geq g_*(x)\}$ denotes the closure of the inaction region $\mathcal{C}$ (cf.\ \eqref{def:CSsets}).
\end{prop}

\begin{theorem}\label{theorem8}
The control $\nu^*$ defined in \eqref{candidateoptimalcontrol} is optimal for problem \eqref{valuefunction} and $F \equiv U$ (cf.~\eqref{def-F}).
\end{theorem}

\begin{proof}
The proof is based on a verification argument and, as usual, it splits into two steps.\vspace{+8pt}

{\em Step 1.}
Fix $(x,c)\in\RR\times[0,1]$ and take $R>0$. Set $\tau_{R}:=\inf\big\{t\ge 0\,:\,X^x_t\notin(-R,R)\big\}$, take an admissible control $\nu$, and recall the regularity results for $F$ of Proposition \ref{listprops}. Then we can use It\^o's formula in its classical form up to the stopping time $\tau_R \wedge T$, for some $T>0$, to obtain
\begin{align*}
F(x,c) =& \EE\left[e^{-\lambda(\tau_{R}\wedge T)}F(X_{\tau_{R}\wedge T}^x, {C}_{\tau_{R}\wedge T}^{c,\nu})\right]-\EE\bigg[\int_0^{\tau_{R}\wedge T}e^{-\lambda s}(\mathbb{L}_X-\lambda)
F(X^x_s,{C}^{c,\nu}_s)ds\bigg] \\
&  - \EE\bigg[\int_0^{\tau_{R}\wedge T}e^{-\lambda s}F_c(X^x_s,{C}^{c,\nu}_s)d\nu_s\bigg]
\\ &  - \EE\left[\sum_{0\leq s < \tau_{R}\wedge T}e^{-\lambda s}
\left(F(X^x_s,{C}^{c,\nu}_{s+})-F(X^x_s,{C}^{c,\nu}_s)-F_c(X^x_s,{C}^{c,\nu}_s)\Delta \nu_s\right)\right]
\end{align*}
where $\Delta \nu_s := \nu_{s+}-\nu_s$ and the expectation of the stochastic integral vanishes since $F_x$ is bounded on $(x,c)\in[-R,R]\times[0,1]$.

Now, recalling that any $\nu \in \mathcal{S}_c$ can be decomposed into the sum of its continuous part and of its pure jump part, i.e.\ $d\nu=d\nu^{cont}+ \Delta \nu$, one has (see \cite{FlemingSoner}, Chapter 8, Section VIII.4, Theorem 4.1 at pp.\ 301-302)
\begin{align*}
F(x,c) = & \EE\left[e^{-\lambda(\tau_{R}\wedge T)}F(X_{\tau_{R}\wedge T}^x, C_{\tau_{R}\wedge T}^{c,\nu})\right]-\EE\bigg[\int_0^{\tau_{R}\wedge T}e^{-\lambda s}(\mathbb{L}_X-\lambda)F(X^x_s,C^{c,\nu}_s)ds\bigg] \\
&  - \EE\bigg[\int_0^{\tau_{R}\wedge T}e^{-\lambda s}F_c(X^x_s,C^{c,\nu}_s)d\nu_s^{cont} - \sum_{0\le s < \tau_{R}\wedge T}e^{-\lambda s}
\left(F(X^x_s,C^{c,\nu}_{s+})-F(X^x_s,C^{c,\nu}_s)\right)\bigg].
\end{align*}
Since $F$ satisfies the HJB equation \eqref{HJB-U} (cf.\ Proposition \ref{prop6}) and by noticing that
\begin{equation}
\label{jump}
F(X^x_s,C^{c,\nu}_{s+})-F(X^x_s,C^{c,\nu}_{s})=
\int_0^{\Delta \nu_s} F_c(X^x_s,C^{c,\nu}_{s}+u)du,
\end{equation}
we obtain
\begin{align}
\label{verif04}
F(x,c) \leq &\EE\left[e^{-\lambda(\tau_{R}\wedge T)}F(X_{\tau_{R}\wedge T}^x, C_{\tau_{R}\wedge T}^{c,\nu})\right]+\EE\bigg[\int_0^{\tau_{R}\wedge T}e^{-\lambda s}\lambda X^x_s \Phi(C^{c,\nu}_s)ds\bigg]\nonumber \\
& + \EE\bigg[\int_0^{\tau_{R}\wedge T}e^{-\lambda s} X^x_s d\nu_s^{cont}\bigg]
 + \EE\left[\sum_{0\leq s < \tau_{R}\wedge T}e^{-\lambda s}
X^x_s \Delta \nu_s \right] \\
 = &\EE\bigg[e^{-\lambda(\tau_{R}\wedge T)}F(X_{\tau_{R}\wedge T}^x, C_{\tau_{R}\wedge T}^{c,\nu})+\int_0^{\tau_{R}\wedge T}e^{-\lambda s}
\lambda X^x_s \Phi(C^{c,\nu}_s)ds + \int_0^{\tau_{R}\wedge T}e^{-\lambda s}X^x_s d\nu_s\bigg].\nonumber
\end{align}

When taking limits as $R\to\infty$ we have $\tau_{R}\wedge T \rightarrow T$, $\PP$-a.s. The integral terms in the last expression on the right-hand side of \eqref{verif04} are uniformly integrable (cf.~\eqref{eq:IJ}) and $F$ has sub-linear growth (cf.~\eqref{boundsF}). Then we also take limits as $T\uparrow \infty$ and it follows
\begin{align}
F(x,c) \leq \EE\bigg[\int_0^{\infty}e^{-\lambda s}\lambda X^x_s \Phi(C^{c,\nu}_s)ds + \int_0^{\infty}e^{-\lambda s} X^x_s d\nu_s \bigg],
\end{align}
due to the fact that $\lim_{T \rightarrow \infty}\EE[e^{-\lambda T}F(X^x_T,C^{c,\nu}_T)]=0$.
Since the latter holds for all admissible $\nu$ we have $F(x,c) \leq U(x,c)$.
\vspace{+8pt}

{\em Step 2.} If $c=1$ then $F(x,1)=U(x,1)=0$. Take then $c\in[0,1)$, $C^*$ as in Proposition \ref{C-opt} and define $\rho:=\inf\big\{t\ge 0\,:\,\nu^*_t = 1-c\big\}$. We can repeat arguments of {Step 1.}~on It\^o's formula with $\tau_R$ replaced by $\tau_R\wedge\rho$ to find
\begin{align*}
F(x,c) =& \EE\left[e^{-\lambda\,(\tau_R\wedge\rho)}F(X_{\tau_R\wedge\rho}^x, C^*_{\tau_R\wedge\rho})\right]-\EE\bigg[\int_0^{\tau_R\wedge\rho}e^{-\lambda s}(\mathbb{L}_X-\lambda)
F(X^x_s,C^*_s)ds\bigg] \\
&  - \EE\bigg[\int_0^{\tau_R\wedge\rho}e^{-\lambda s}F_c(X^x_s,C^*_s)d\nu_s^{*,cont}\bigg]\\
& - \EE\left[\sum_{0\le s < \tau_R\wedge\rho}e^{-\lambda s}
\left(F(X^x_s,C^*_{s+})-F(X^x_s,C^*_s)\right)\right].
\end{align*}
If we now recall Proposition \ref{prop6}, Proposition \ref{C-opt} and \eqref{jump}, then from the above we obtain
\begin{align}
\label{verif07}
F(x,c) =& \EE\bigg[e^{-\lambda\,(\tau_R\wedge\rho)}F(X_{\tau_R\wedge\rho}^x, C^*_{\tau_R\wedge\rho})+\int_0^{\tau_R\wedge\rho}e^{-\lambda s}
\lambda X^x_s \Phi(C^*_s)ds + \int_0^{\tau_R\wedge\rho}e^{-\lambda s}X^x_s d\nu^*_s\bigg]
\end{align}
As $R\to\infty$, again $\tau_R\to\infty$, clearly $\tau_R\wedge\rho\to \rho$, $\PP$-a.s.\ and $\EE\left[e^{-\lambda(\tau_R\wedge\rho)}F(X_{\tau_R\wedge\rho}^x, C^*_{\tau_R\wedge\rho})\right]\to 0$. Moreover, we also notice that since $d\,\nu^*_s\equiv0$ and $\Phi(C^*_s)\equiv0$ for $s>\rho$ the integrals in the last expression of \eqref{verif07} may be extended beyond $\rho$ up to $+\infty$ so as to obtain
\begin{align}
\label{verif08}
F(x,c) =& \EE\bigg[\int_0^{\infty}e^{-\lambda s}
\lambda X^x_s \Phi(C^*_s)ds + \int_0^{\infty}e^{-\lambda s}X^x_s d\nu^*_s\bigg]=\mathcal{J}_{x;c}(\nu^*).
\end{align}
Then $F \equiv U$ and $\nu^*$ is optimal.
\end{proof}

\begin{figure}[!ht]
\centering
\includegraphics[scale=0.5]{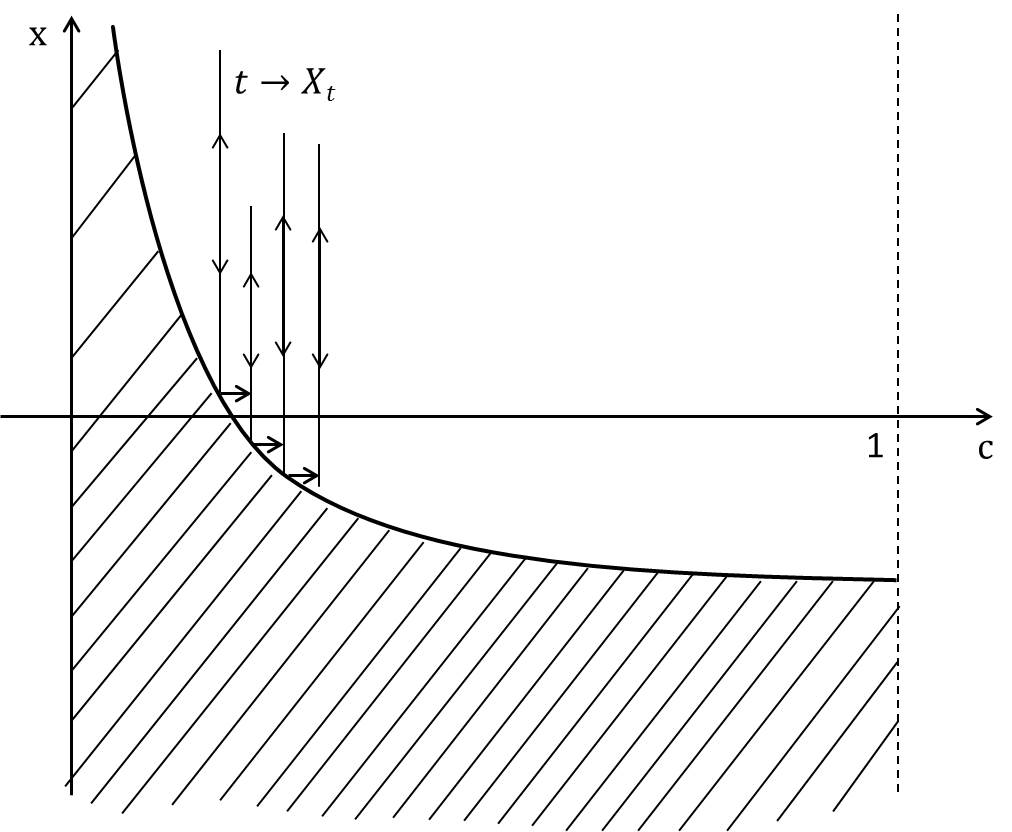}
\caption{\small An illustrative diagram of the action/inaction regions in the case $\hat{c}\le0$ and of the optimal control $\nu^*$ (see \eqref{candidateoptimalcontrol}). The boundary $\beta_*$ splits the state space into the inaction region (white) and action region (hatched). When the initial state is $(x,c)$ with $x>\beta_*(c)$ one observes a Skorokhod reflection of $(X^x,C^{c,\nu^*})$ at $\beta_*$ in the horizontal direction up to when all the fuel is spent.}\label{fig:1}
\end{figure}


\section{The Case $\hat{c} \geq 1$}
\label{Case3}

In this section we examine the opposite regime to that of Section \ref{Case2}, when the infimum in \eqref{SSCDS} is attained by an almost surely finite stopping time $\tau^*$ and the constant control policy $\hat{\nu} \equiv 0$. Equivalently the solution to \eqref{valuefunction} does not exert control before the price process $X$ hits a repelling boundary, at which point all available control is exerted. We show that this regime occurs when $k(c)<0$ for all $c\in[0,1]$ (cf.\ \eqref{def-k}), or equivalently $\hat{c} > 1$. We confirm 
this contrast with the differential relationship holding in Section \ref{Case2} (i.e.~the break-down of the classical connection to optimal stopping) by showing that the principle of smooth fit does not hold for the value function of the control problem, whose second order mixed derivative $U_{cx}$ is not continuous across the optimal boundary. The case $\hat c = 1$ is similar, see Remark \ref{rem:gammalim}.

We begin by observing that exercising no control produces a payoff equal to
\begin{align}\label{heur01}
\lambda\Phi(c)\int^\infty_0{e^{-\lambda s}\EE\left[X^x_s\right]ds}
\end{align}
(cf.~\eqref{valuefunction}). Suppose instead that we exert a small amount $\Delta^0$ of control at time zero and exercise no further control. In this case the cost of control is $x\Delta^0$ and, approximating $\Phi(c+\Delta^0)\sim\Phi(c)+\Phi'(c)\Delta^0$, the payoff reads
\begin{align}\label{heur02}
&\lambda\Phi(c)\int^\infty_0{e^{-\lambda s}\EE\left[X^x_s\right]ds}+\Delta^0\lambda\Phi'(c)\int^\infty_0{e^{-\lambda s}\EE\big[X^x_s\big]ds}+x\Delta^0\nonumber\\
&=\lambda\Phi(c)\int^\infty_0{e^{-\lambda s}\EE\left[X^x_s\right]ds}+\frac{
\Delta^0}{\lambda+\theta}\big(k(c)x+\theta\mu\Phi'(c)\big)
\end{align} recalling that $\EE[X^x_s]=\mu+(x-\mu)e^{-\theta s}$ (cf.~\eqref{OUexplicit}) to obtain the second term. Comparing \eqref{heur01} and \eqref{heur02} we observe that
the relative net contribution to the infimum \eqref{valuefunction} (equivalently, the infimum \eqref{SSCDS}) from exercising  the amount $\Delta^0$ of control is given by the second term in the second line of \eqref{heur02}, which for fixed $c$ depends only on the term $k(c)x$. When $x>-\theta \mu\Phi'(c)/k(c)$ the second term in \eqref{heur02} is negative and therefore favourable, while when
$x<-\theta \mu\Phi'(c)/k(c)$ it is positive and unfavourable. 
This suggests that in the present case, when $\hat{c} > 1$, we should expect the inaction region to correspond to $\{(x,c): x < \gamma(c)\}$ for some function $\gamma$. Moreover, since the curve $c \mapsto -\theta \mu\Phi'(c)/k(c)$ is strictly decreasing as $\Phi$ is strictly convex, small control increments in this profitable region $x>-\theta \mu\Phi'(c)/k(c)$ keep the state process $(X,C)$ inside the same region. It thus follows that infinitesimal increments due to a possible reflecting boundary as in Section \ref{Case2} do not seem to lead to an optimal strategy. Instead a phenomenon similar to `critical depensation' in optimal harvesting models is suggested, where it becomes optimal to exercise all available control upon hitting a repelling free boundary  (see for example \cite{Alvarez98} for one-dimensional problems but note that in our setting the free boundary will in general be non-constant). 

We solve the optimisation problem \eqref{valuefunction} by directly tackling the associated Hamilton-Jacobi-Bellmann equation suggested by the above heuristic and the dynamic programming principle. 
It is not difficult to show from \eqref{valuefunction} that $x\mapsto U(x,c)$ has at most sub-linear growth:
indeed, integrating by parts the cost term $\int^{\infty}_0 {e^{-\lambda s}X^x_s\,d{\nu}_s}$ and noting that the martingale $M_t := \int_0^{t}\sigma e^{-\lambda s}\nu_s dB_s$ is uniformly integrable, we can write for any $\nu \in \mathcal{S}_c$
\begin{align*}
\mathcal{J}_{x,c}(\nu)  \leq \EE\bigg[\int_0^{\infty}e^{-\lambda s}\Big(\lambda |X^x_{s}|\Phi(
C^{c,\nu}_{s}) + |\nu_s| [\lambda|X^x_s| + \theta(\mu + |X^x_s|)]\Big) ds\bigg]   \leq   K(1 + |x|),
\end{align*}
for some suitable $K>0$, by \eqref{OUexplicit}, Assumption \ref{ass-Phi} and the fact that any admissible $\nu$ is nonnegative and uniformly bounded.

We seek a couple $(W,\gamma)$ solving the following system
\begin{align}
\label{HJB-case3}
\left\{
\begin{array}{ll}
\mathbb{L}_X W(x,c)-\lambda W(x,c)=-\lambda\,x\,\Phi(c), & \text{for $x<\gamma(c)$, $c\in[0,1)$},\\[+3pt]
W_c(x,c)\ge -x, &\textrm{for $(x,c)\in\mathbb{R}\times[0,1]$},\\[+3pt]
W(x,c)=x(1-c), & \text{for $x\ge \gamma(c)$, $c\in[0,1]$},\\[+3pt]
W_x(\gamma(c),c)=(1-c), & \text{for $c\in[0,1]$}.
\end{array}
\right.
\end{align}
We will verify a posteriori that $W$ then also satisfies $W_c(\gamma(c),c)=-\gamma(c)$ but does not satisfy $W_{cx}(\gamma(c),c)=-1$, which is a smooth fit condition often employed in the solution of singular stochastic control problems (see, for example, \cite{FedericoPham} and \cite{MehriZervos}).


\begin{theorem}
\label{thm-HJB}
Let $\psi_\lambda$ be the increasing fundamental solution of $(\mathbb{L}_X-\lambda)f=0$ (cf.~\eqref{psi} in Appendix) and define
\begin{align}
\label{over-x0}
\overline{x}_0(c):= \frac{\theta\mu\Phi(c)}{\zeta(c)}, \quad c \in [0,1),
\end{align}
where $\zeta(c):=(\lambda + \theta)(1-c) - \lambda\Phi(c) = \int^1_c k(y)dy < 0.$
There exists a unique couple $(W, \gamma)$ solving \eqref{HJB-case3} with $W$ satisfying $W\in W^{2,1,\infty}_{loc}(\mathbb{R}\times(0,1))$ and $W_c(\gamma(c),c)=-1$. The function $\gamma$ is decreasing and, if $\hat{c}>1$, it is $C^1$ on $[0,1]$. For each $c\in[0,1]$, $\gamma(c)\in (\overline{x}_0(c),+\infty)$ is the unique solution of
\begin{align}\label{freebeq}
\text{find $x\in\RR$:}\quad\frac{\psi_{\lambda}(x)}{\psi'_{\lambda}(x)}= x - \overline{x}_0(c) .
\end{align}
For $c\in[0,1]$ the function $W$ may be expressed in terms of $\gamma$ as
\begin{align}\label{def-F02}
W(x,c)\hspace{-2pt}=\hspace{-3pt}\left\{\hspace{-4pt}
\begin{array}{ll}
\frac{\psi_{\lambda}(x)}{\psi_{\lambda}(\gamma(c))}\Big[\gamma(c)(1\hspace{-2pt}-\hspace{-2pt}c) \hspace{-2pt}- \hspace{-2pt}\lambda\Phi(c)\hspace{-2pt}
\left(\frac{\gamma(c)-\mu}{\lambda+\theta}\hspace{-2pt}+\hspace{-2pt}\frac{\mu}{\lambda}\right)\hspace{-2pt}\Big]\hspace{-2pt} +\hspace{-2pt}\lambda\Phi(c)\hspace{-2pt}
\left[\frac{x-\mu}{\lambda+\theta}\hspace{-2pt}+\hspace{-2pt}\frac{\mu}
{\lambda}\right], & \text{for $x<{\gamma(c)}$,} \\[+10pt]
\displaystyle x(1-c), & \text{for $x\ge \gamma(c)$. }
\end{array}
\right.
\end{align}
Moreover the map $x\mapsto W_c(x,c)$ is not $C^1$ across the boundary $\gamma$ and one has $W_{cx}(\gamma(c),c) < -1$, $c\in[0,1]$.
\end{theorem}
\begin{proof}
The proof will be carried out in several steps.
\vspace{+8pt}

\emph{Step 1.} The first equation in \eqref{HJB-case3} is an ordinary differential equation solved by
\begin{align}
\label{candF01}
W(x,c)=A(c)\psi_{\lambda}(x)+B(c)\phi_{\lambda}(x)+\lambda\,\Phi(c)\left[\frac{x-\mu}{\lambda+\theta}+\frac{\mu}
{\lambda}\right],
\end{align}
with $\phi_{\lambda}$ and $\psi_{\lambda}$ as in \eqref{phi} and \eqref{psi}, respectively. Since $W(x,c)=x(1-c)$ for $x>\gamma(c)$ sub-linear growth is fulfilled as $x\to+\infty$; however, as $x\to-\infty$ one has that $\phi_{\lambda}(x)\to+\infty$ with a super-linear trend. Since we are trying to identify $U$, it is then natural to set $B(c)\equiv0$. Imposing the third and fourth conditions of \eqref{HJB-case3} at $x=\gamma(c)$ we find
\begin{align}\label{Ac}
A(c)\psi_{\lambda}(\gamma(c))= 
\gamma(c)(1-c) - \lambda\,\Phi(c)
\left(\frac{\gamma(c)-\mu}{\lambda+\theta}+\frac{\mu}{\lambda}\right)
\end{align}
and
\begin{align}\label{Ac2}
\quad A(c)\psi'_{\lambda}(\gamma(c))= (1-c) - \frac{\lambda\Phi(c)}{\lambda+\theta}.
\end{align}
from which it
follows that $\gamma(c)$ should solve \eqref{freebeq}. Since $\psi_{\lambda}/\psi'_{\lambda} > 0$ any solution of \eqref{freebeq} must be in the set $(\overline{x}_0(c),+\infty)$ and so \eqref{freebeq} is equivalent to
finding $x\in(\overline{x}_0(c),+\infty)$ such that $\overline{H}(x,c)=0$ with
\begin{align}
\label{def-H03}
\overline{H}(x,c):=\psi_{\lambda}(x)\Big[(1-c) - \frac{\lambda\Phi(c)}{\lambda+\theta}\Big]-\psi'_{\lambda}(x)\Big[x(1-c) - \lambda\,\Phi(c)
\left(\frac{x-\mu}{\lambda+\theta}+\frac{\mu}{\lambda}\right)\Big].
\end{align}
Since $\psi'_{\lambda}>0$ and $\psi''_{\lambda}>0$ (cf.\ \eqref{psi} and \eqref{cylinder}) it follows by direct calculation that $\overline{H}_x(x,c)>0$ and $\overline{H}_{xx}(x,c)>0$ on $x\in(\overline{x}_0(c),+\infty)$; moreover, since $\overline{H}(\overline{x}_0(c),c)<0$ there exists a unique $\gamma(c)$ solving \eqref{freebeq}. Now from \eqref{freebeq}, \eqref{Ac} and  \eqref{Ac2} we can equivalently set
\begin{eqnarray}
\label{def-Ac}
A(c) &\hspace{-0.25cm}:= \hspace{-0.25cm} &\frac{1}{\psi_{\lambda}(\gamma(c))}\Big[\gamma(c)(1-c) - \lambda\,\Phi(c)
\left(\frac{\gamma(c)-\mu}{\lambda+\theta}+\frac{\mu}{\lambda}\right)\Big] \nonumber \\
&\hspace{-0.25cm} = \hspace{-0.25cm}&\frac{1}{\psi'_{\lambda}(\gamma(c))}\Big[(1-c) - \frac{\lambda\Phi(c)}{\lambda+\theta}\Big]
\end{eqnarray}
and \eqref{def-F02} follows by extending $W$ to be $x(1-c)$ for $x>\gamma(c)$.
\vspace{+8pt}

\emph{Step 2.} Using \eqref{freebeq} and \eqref{def-F02} it is easy to check that $W(\gamma(c),c)=\gamma(c)(1-c)$ and $W_x(\gamma(c),c)=(1-c)$.
\vspace{+8pt}

\emph{Step 3.} In order to establish the monotonicity of $\gamma$ we study the derivative with respect to $c$ of the map $c\mapsto x - \overline{x}_0(c)$. Differentiating we obtain
\begin{align}\label{monot01}
\frac{d}{d\,c}(x - \overline{x}_0(c))= - \frac{d}{d\,c}\overline{x}_0(c) =
-\frac{\theta\mu(\lambda+\theta)[\Phi'(c)(1-c) + \Phi(c)]}{\zeta^2(c)} > 0,
\end{align}
where the last inequality holds since $-\Phi(c)=\int^1_c{\Phi'(y)dy}>\Phi'(c)(1-c)$ by strict convexity of $\Phi$. Now \eqref{monot01} guarantees that $c\mapsto x - \overline{x}_0(c)$ is increasing and then the implicit function theorem and arguments similar to those that led to \eqref{monotbeta} in the proof of Proposition \ref{bstar01} allow us to conclude that $\gamma$ lies in $C^1([0,1])$ (if $\hat{c}>1$) and is decreasing (for $\hat{c}=1$ see Remark \ref{rem:gammalim} below).
\vspace{+8pt}

\emph{Step 4.} We now aim to prove the second condition in \eqref{HJB-case3}. Recalling that $W$ has been extended to be $x(1-c)$ for $x\ge \gamma(c)$ the result is trivial in that region. Consider only $x<\gamma(c)$. From \eqref{def-F02} we have
\begin{eqnarray}
\label{Fc-cont01}
& & W(x,c)=x(1-c)- \Big[x(1-c) - \lambda\,\Phi(c)\left(\frac{x-\mu}{\lambda+\theta}+\frac{\mu}{\lambda}\right)\Big] \nonumber \\
& & \hspace{1.8cm} +\frac{\psi_{\lambda}(x)}{\psi_{\lambda}(\gamma(c))}\Big[\gamma(c)(1-c) - \lambda\,\Phi(c)
\left(\frac{\gamma(c)-\mu}{\lambda+\theta}+\frac{\mu}{\lambda}\right)\Big]
\end{eqnarray}
and since $\gamma$ is differentiable, recalling \eqref{Gexpression} and rearranging terms we have
\begin{align}
\label{Fc-cont02}
W_c(x,c)=&-x+G(x,c)-\frac{\psi_{\lambda}(x)}{\psi_{\lambda}(\gamma(c))}\,G(\gamma(c),c)\nonumber\\
&+\frac{\psi_{\lambda}(x)}{\psi_{\lambda}(\gamma(c))}\,\gamma'(c)\,\Big[(1-c) -\frac{\lambda\Phi(c)}{\lambda+\theta}\Big]\Big[1 - \frac{\psi_{\lambda}'(\gamma(c))}{\psi_{\lambda}(\gamma(c))}(\gamma(c) - \overline{x}_0(c))\Big] \\
=&-x+G(x,c)-\frac{\psi_{\lambda}(x)}{\psi_{\lambda}(\gamma(c))}\,G(\gamma(c),c),\nonumber
\end{align}
where the last equality follows since $\gamma$ solves \eqref{freebeq}. Note that as a byproduct of \eqref{Fc-cont02} we also have $W_c(\gamma(c),c)=-\gamma(c)$.
Differentiating \eqref{Fc-cont02} with respect to $x$ and taking $x=\gamma(c)$ gives
\begin{align}\label{Fc-cont03}
W_{cx}(\gamma(c),c)+1= \frac{k(c)}{\lambda+\theta} - \frac{\psi'_{\lambda}(\gamma(c))}{\psi_{\lambda}(\gamma(c))}G(\gamma(c),c)
\end{align}
and hence from \eqref{Gexpression} and \eqref{freebeq} we obtain
\begin{eqnarray}
\label{Fc-cont04}
 W_{cx}(\gamma(c),c)+1 &\hspace{-0.25cm} = \hspace{-0.25cm} & - \frac{k(c)}{(\lambda + \theta)}\frac{1}{(\gamma(c) - \overline{x}_0(c))}
\Big[\frac{\mu\theta\Phi'(c)}{k(c)} + \overline{x}_0(c)\Big] \nonumber \\
& \hspace{-0.25cm} = \hspace{-0.25cm} &- \frac{\theta\mu}{\zeta(c)}\frac{1}{(\gamma(c) - \overline{x}_0(c))}[\Phi'(c)(1-c) + \Phi(c)].
\end{eqnarray}
Since $\gamma(c)> \overline{x}_0(c)$, $\zeta(c) < 0$ and $\Phi'(c)(1-c) + \Phi(c)<0$ by the convexity of $\Phi$ we conclude that
\begin{align}\label{Fc-cont05}
W_{cx}(\gamma(c),c)+1<0,\qquad c\in[0,1].
\end{align}

\noindent For $x<\gamma(c)$ we can differentiate with respect to $c$ and $x$ the first equation in \eqref{HJB-case3}, set $\bar{u}(x,c):=W_{xc}(x,c)+1$ and find
\begin{align}
\label{Fc-cont06}
\mathbb{L}_X\bar{u}(x,c)-(\lambda+\theta)\bar{u}(x,c)=-k(c) \geq 0,\quad\text{for $c\in[0,1]$ and $x<\gamma(c)$,}
\end{align}
with boundary condition $\bar{u}(\gamma(c),c)=W_{xc}(\gamma(c),c)+1<0$. Taking $\sigma_\gamma:=\inf\big\{t\ge 0\,:\,X^x_t\ge \gamma(c)\big\}$ and using It\^o's formula we find
\begin{equation}
\label{Fc-cont07}
\bar{u}(x,c)=\EE\left[e^{-(\lambda+\theta)\sigma_\gamma}\bar{u}\big(X^x_{\sigma_\gamma},c\big)+k(c)\int^{\sigma_\gamma}_0
{e^{-(\lambda+\theta)s}ds}\right],\quad\mbox{for $c\in[0,1]$ and $x<\gamma(c)$}.
\end{equation}
It follows from \eqref{Fc-cont02} and recurrence of $X$ that $e^{-(\lambda+\theta)\sigma_\gamma}\bar{u}\big(X^x_{\sigma_\gamma},c\big)=
e^{-(\lambda+\theta)\sigma_\gamma}\bar{u}\big(\gamma(c),c\big)$, $\PP$-a.s. Moreover, $k(c)<0$ and \eqref{Fc-cont05} imply that the right-hand side of \eqref{Fc-cont07} is strictly negative. It follows that $W_{xc}(x,c)+1<0$ for all $x<\gamma(c)$ and hence $x\mapsto W_c(x,c)+x$ is decreasing. Since $W_c(\gamma(c),c)+\gamma(c)=0$ by \eqref{Fc-cont02}, then we can conclude $W_c(x,c)+x\ge 0$ for all $(x,c)\in\mathbb{R}\times[0,1]$.
\end{proof}
\begin{remark}
\label{smoothfitbreaksdown}
If $W_c$ were the value function of an optimal stopping problem with free boundary $\gamma$ we would expect the principle of smooth fit to hold, i.e.~$W_c(\,\cdot\,,c)\in C^1$ across the boundary $\gamma$. In the literature on singular stochastic control, continuity of $W_{cx}$ is usually verified (cf.~for instance \cite{FedericoPham} and \cite{MehriZervos}) and often used to characterise the optimal boundary. However equation \eqref{Fc-cont05} confirms that this property does not hold in this example, and indeed the differential relationship between singular control and optimal stopping (in the sense, e.g., of \cite{ElKK88}, \cite{KaratzasShreve84}, \cite{K85}) breaks down.
\end{remark}
\begin{remark}
\label{rem:gammalim}
It is interesting to note that if $\hat{c}=1$ one has $\lim_{c\uparrow\hat{c}}\overline{x}_0(c)=-\infty$ and hence $\lim_{c\uparrow\hat{c}}\gamma(c)=-\infty$, otherwise a contradiction is found when passing to the limit in \eqref{freebeq} with $x=\gamma(c)$.
\end{remark}
Since $\gamma$ solving \eqref{freebeq} is the unique candidate optimal boundary we set $\gamma_*:=\gamma$ from now on.
\begin{prop}\label{propHJB-case3}
The function $W$ of Theorem \ref{thm-HJB} solves \eqref{HJB-U} with $W(x,1)=U(x,1)=0$.
\end{prop}
\begin{proof}
The boundary condition at $c=1$ follows from \eqref{def-F02}. Since $W$ solves \eqref{HJB-case3} it also solves \eqref{HJB-U} for $x<\gamma_*(c)$, $c\in[0,1]$. It thus remains to prove that $\big(\mathbb{L}_X-\lambda\big)W(x,c)\ge -\lambda\Phi(c)\,x$ for $x>\gamma_*(c)$. Note that since $W(x,c)=x(1-c)$ in that region then $\big(\mathbb{L}_X-\lambda\big)W(x,c)=(1-c)\big[\theta\mu-(\lambda+\theta)x\big]$. Set
\begin{align}
\label{xtilde}
\tilde{x}(c):=\frac{(1-c)\theta\mu}{\zeta(c)}, \quad c\in [0,1),
\end{align}
where again $\zeta(c)=\int_c^1k(y)dy$, and observe that $(1-c)\big[\theta\mu-(\lambda+\theta)x\big]\ge-\lambda\Phi(c)\,x$ for all $x\ge\tilde{x}(c)$.
To conclude we need only show that $\gamma_*(c)>\tilde{x}(c)$ for $c\in[0,1]$. It suffices to prove that $\overline{H}(\tilde{x}(c),c)<0$ (cf.~\eqref{def-H03}) and the result will follow since $\overline{H}(\cdot,c)$ is strictly increasing and such that $\overline{H}(\gamma_*(c),c)=0$.

Fix $c\in[0,1)$ and denote $\tilde{x}:=\tilde{x}(c)$ and $\overline{x}_0:=\overline{x}_0(c)$ (cf.~\eqref{over-x0}) for simplicity. Then we have
\begin{align}
\label{test01}
\frac{\psi_{\lambda}(\tilde{x})}{\psi'_{\lambda}(\tilde{x})}- (\tilde{x} - \overline{x}_0)
=\frac{\psi_{\lambda}(\tilde{x})}{\psi'_{\lambda}(\tilde{x})} - \frac{\theta\mu}{\zeta(c)}(1-c-\Phi(c))
=\frac{\psi_{\lambda}(\tilde{x})}{\psi'_{\lambda}(\tilde{x})}-\frac{\tilde{x}}
{(1-c)}\left[(1-c)-\Phi(c)\right],
\end{align}
where the last equality follows from \eqref{xtilde}.
Since $\psi''_{\lambda}>0$ and $\psi_{\lambda}$ solves $\big(\mathbb{L}_X-\lambda\big)\psi_{\lambda}=0$ we obtain
\begin{align}\label{test02}
\frac{\psi_{\lambda}(\tilde{x})}{\psi'_{\lambda}(\tilde{x})}>\frac{\theta(\mu-\tilde{x})}{\lambda}
\end{align}
and from the right-hand side of \eqref{test01} also
\begin{align}\label{test03}
\frac{\psi_{\lambda}(\tilde{x})}{\psi'_{\lambda}(\tilde{x})}-(\tilde{x} - \overline{x}_0)>&\frac{\theta(\mu-\tilde{x})}{\lambda}-\frac{\tilde{x}
\left[\lambda(1-c)-\lambda\Phi(c)\right]}
{\lambda(1-c)}\nonumber\\
=&\frac{\big(\theta\mu-(\lambda+\theta)\tilde{x}\big)(1-c)+\lambda\Phi(c)\tilde{x}}{\lambda(1-c)}=0.
\end{align}
The inequality above implies $\overline{H}(\tilde{x}(c),c)<0$ so that $\gamma_*(c)>\tilde{x}(c)$. Hence $\big(\mathbb{L}_X-\lambda\big)W(x,c)\ge -\lambda\Phi(c)\,x$ for $x>\gamma_*(c)$.

\end{proof}

Introduce the stopping time
\beq
\label{optimaltimecase3}
\tau_*:=\inf\big\{t\ge 0\,:\,X^x_t\ge \gamma_*(c)\big\},
\eeq
and for any $c\in[0,1)$ define the admissible control strategy
\begin{align}
\label{op-contr01}
\nu^*_t:=\left\{
\begin{array}{ll}
0, & t\le \tau_*,\\
(1-c), & t>\tau_*.
\end{array}
\right.
\end{align}
\begin{figure}[!ht]
\centering
\includegraphics[scale=0.5]{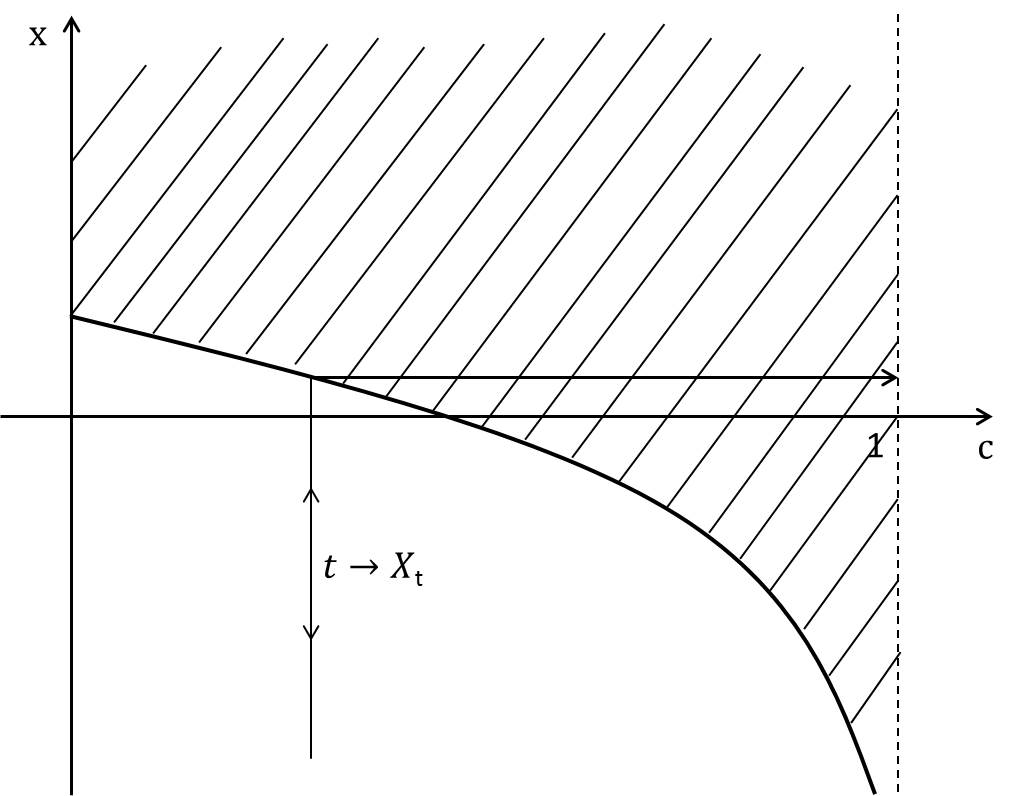}
\caption{\small An illustrative diagram of the action/inaction regions in the case $\hat{c}>0$ and of the optimal control $\nu^*$ (see \eqref{op-contr01}). The boundary $\gamma_*$ splits the state space into the inaction region (white) and action region (hatched). When the initial state is $(x,c)$ with $x<\gamma_*(c)$, the first time $X^x$ hits $\gamma_*(c)$ one observes a single jump of $(X^x,C^{c,\nu^*})$ in the horizontal direction up to $c=1$.}\label{fig:2}
\end{figure}
\begin{theorem}
\label{thm-opt-c}
The admissible control $\nu^*$ of \eqref{op-contr01} is optimal for problem \eqref{valuefunction} and $W\equiv U$.
\end{theorem}
\begin{proof}
The proof employs arguments similar to those used in the proof of Theorem \ref{theorem8}. We recall the regularity of $W$ by Theorem \ref{thm-HJB} and note that $\big|W(x,c)\big|\le K(1+|x|)$ for a suitable $K>0$. Then an application of It\^o's formula in the weak version of \cite{FlemingSoner}, Chapter 8, Section VIII.4, Theorem 4.1, easily gives $W(x,c)\le U(x,c)$ for all $(x,c)\in\mathbb{R}\times[0,1]$ (cf.~also arguments in step 1 of the proof of Theorem \ref{theorem8}).

On the other hand, taking $C^*_t:=C^{c,\nu^*}_t=c+\nu^*_t$, $c\in[0,1)$, with $\nu^*$ as in \eqref{op-contr01}, and applying It\^o's formula again (possibly using localisation arguments as in the proof of Theorem \ref{theorem8}) we find
\begin{align}
\label{Ito}
W(x,c)=&\EE\left[e^{-\lambda\tau_*}W(X^x_{\tau_*},C^*_{\tau_*})+\int^{\tau_*}_0{e^{-\lambda s}\lambda\,X^x_s\,
\Phi(C^*_{s})ds}\right]\nonumber\\
&-\EE\left[\int^{\tau_*}_0{e^{-\lambda s}W_c(X^x_s,C^*_{s})d\,\nu^{*,cont}_s}\right]\\
&-\EE\left[\sum_{0\le s<\tau_*}e^{-\lambda s}\big(W(X^x_s,C^*_{s+})-W(X^x_s,C^*_{s})\big)\right].\nonumber
\end{align}
Since $(X^x_s,C^*_s)=(X^x_s,c)$ for $s\le\tau_*$, then the third and fourth term on the right-hand side of \eqref{Ito} equal zero, whereas for the first term we have from \eqref{HJB-case3} and \eqref{op-contr01}
\begin{align}\label{opt-C01}
\EE&\left[e^{-\lambda\tau_*}W(X^x_{\tau_*},c+\nu^*_{\tau_*})\right]=\EE\left[e^{-\lambda\tau_*}
W(X^x_{\tau_*},c)\right]\nonumber\\
=&\EE\left[e^{-\lambda\tau_*}X^x_{\tau_*}(1-c)\right]=\EE\left[\int^\infty_0{e^{-\lambda s}X^x_sd\,C^*_s}\right].
\end{align}
For the second term on the right-hand side of \eqref{Ito} we have
\begin{align}\label{opt-C02}
\EE\left[\int^{\tau_*}_0{e^{-\lambda s}\lambda\,X^x_s\,
\Phi(c+\nu^*_{s})ds}\right]=\EE\left[\int^{\infty}_0{e^{-\lambda s}\lambda\,X^x_s\,
\Phi(c+\nu^*_{s})ds}\right],
\end{align}
since $\Phi(1)=0$ by Assumption \ref{ass-Phi}. Now, \eqref{Ito}, \eqref{opt-C01} and \eqref{opt-C02} give $W(x,c) = U(x,c)$, and $C^*$ is optimal.
\end{proof}

\section{Considerations in the case $\hat{c} \in (0,1)$}
\label{CompleteCase}

In this section we discuss the remaining case when $\hat{c}\in(0,1)$, equivalently when the function $k(\cdot)$ of \eqref{def-k} changes its sign over $(0,1)$.\vspace{+6pt}

$1.$ For $c\in[\hat{c},1]$ it can be seen that setting the strictly convex penalty function in \eqref{nonconvex} equal to $\hat{\Phi}(\,\cdot\,):=\Phi(\hat{c}+\,\cdot\,)$ reduces problem \eqref{valuefunction} to that of Section \ref{Case2}. The optimal control strategy for $c\in[\hat{c},1]$ is therefore of reflecting type and it is characterised in terms of a decreasing boundary $\hat{\beta}$ defined on $(\hat{c},1]$. As expected the classical connection with optimal stopping holds in the sense that $U_c=v$ on $\RR\times(\hat{c},1]$ with $v$ as in \eqref{opt-st}.\vspace{+6pt}

$2.$ When $c\in[0,\hat{c})$ 
the optimal policy depends both on the local considerations discussed at the beginning of Section \ref{Case3} and also on the solution for $c\in[\hat{c},1]$ given in point $1$ above. Assuming that the analytic expression of $U$ is known for $c\in[\hat{c},1]$ then the HJB equation in the set $\RR\times [0,\hat{c})$ has a natural boundary condition at $c=\hat c$ and its solution is expected to paste (at least) continuously with $U(\,\cdot\,,\hat{c})$.
Since the expression for $U$ obtained in Section \ref{Case2} is non-explicit in general, analysis of the geometry of the action and inaction regions is more challenging in this case and its rigorous study is beyond the scope of this paper; nevertheless we will discuss some qualitative ideas based on the findings of the previous sections. 
\vspace{+6pt}

$3.$ The local considerations at the beginning of Section \ref{Case3} hold in the same way in this setting and therefore we may expect a repelling behaviour of the boundary of the action region. Conjecturing the existence of a decreasing free boundary $\hat{\gamma}$ defined on $[0,\hat{c})$ two possible optimal controls can be envisioned, depending on the position of $\hat{\gamma}$ relative to $\hat{\beta}$ in the $(x,c)$-plane. For an initial inventory $c\in[0,\hat{c})$, once the uncontrolled diffusion $X$ hits $\hat{\gamma}(c)$ the inventory should be increased as follows: $i)$ if $\hat{\gamma}(c)\le\inf_{c\in(\hat{c},1]}\hat{\beta}(c)$ all available control is exerted, otherwise $ii)$ the inventory is increased just enough to push $(X,C)$ inside the portion of inaction region in $\RR\times(\hat{c},1]$ (i.e.~the subset of $\RR\times(\hat{c},1]$ bounded below by $\hat{\beta}$). As a result the optimal boundaries $\hat{\beta}$ and $\hat{\gamma}$ exhibit a strong coupling which together with the difficulty in handling the expressions for $\phi_\lambda$ and $\psi_\lambda$ challenges the methods of solution employed in this paper.\vspace{+6pt}

$4.$ We note that determining the geometry of two coexisting free boundaries in a two dimensional state space is not a novelty in the context of SSCDS but explicit solutions can only be found in some specific models (see~for instance \cite{KOWZ} where a Brownian motion and a quadratic cost are considered).
Indeed it is possible to provide a solution when $\theta=0$, for which we refer the reader to \cite{DeAFeMo14b}.
Before concluding this section we show that the latter results are consistent with the above ideas.
In \cite{DeAFeMo14b} the interval $[0,1]$ for the values of the inventory is again split into two subintervals by a point that here we denote by $\tilde{c}$ for clarity (in \cite{DeAFeMo14b} it is denoted by $\hat{c}$). In the portion $\RR\times[0,\tilde{c})$ of the state space of \cite{DeAFeMo14b} the boundary of the action region is of repelling type consistent with point 3 above, although in this case two repelling boundaries are present.
For $c\in(\tilde{c},1]$ the free boundary in \cite{DeAFeMo14b} is constant with respect to $c$ and, although the optimal policy is therefore of bang-bang type, it is not difficult to see that it may equally be interpreted as the limit of reflecting boundaries. Indeed smooth fit holds at this boundary when $c\in(\tilde{c},1]$, along with the differential  connection with optimal stopping (see p.~3 in the Introduction of \cite{DeAFeMo14b} and Remark 3.3 therein) so that the qualitative behaviour is the same as that described in point $1$ above.

\appendix

\section{A Problem of Storage and Consumption}
\label{app:formulation}
\renewcommand{\theequation}{A-\arabic{equation}}

A problem naturally arising in the analysis of power systems is the optimal charging of electricity storage. We consider the point of view of an agent that commits to fully charging an electrical battery on or before a randomly occurring time $\tau>0$ of demand. At any time $t>0$ prior to the arrival of the demand the agent may increase the storage level $C_t$ (within the limits of its capacity, which is one unit) by buying electricity at the spot price $X_t$. Several specifications of the spot price dynamics can be considered. We take $(X_t)_{t\ge0}$ as a continuous, strong Markov process adapted to a filtration $(\cF_t)_{t\ge0}$ on a complete probability space $(\Omega,\cF,\PP)$.

If the battery is not full at time $\tau$ then it is filled by a less efficent method so that the terminal spot price is weighted by a strictly convex function $\Phi$, and so is equal to $\Psi(X_\tau,C_\tau)=X_\tau\Phi(C_\tau)$ with $\Phi(1)=0$ (cf.~Assumption \ref{ass-Phi}). The storage level can only be increased and the process $C_t = c + \nu_t$ follows the dynamics \eqref{ControlledY} with $\nu\in \cS_c$ (cf.~\eqref{admissiblecontrols}).
For simplicity and with no loss of generality we assume that costs are discounted at a rate $r=0$.

The aim of the agent is to minimise the future expected costs by optimally increasing the storage within its limited capacity. Then the agent faces the optimisation problem with random maturity
\begin{align}\label{vfun00}
\inf_{\nu\in\cS_c}\EE\Big[\int^\tau_0{X_t}d\nu_t+X_\tau\Phi(C_\tau)\Big].
\end{align}
Various specifications for the law of $\tau$ are clearly possible. Here we consider only the case of $\tau$ independent of the filtration $(\cF_t)_{t\ge0}$ and distributed according to an exponential law with parameter $\lambda>0$; that is,
\begin{align}
\PP\big(\tau>t\big)=e^{-\lambda t}.
\end{align}
This setting effectively models the demand as completely unpredictable. By the assumption of independence of $\tau$ and $(X,C)$, for any $\nu$ we easily obtain
\begin{align}
\EE\Big[X_\tau\Phi(C_\tau)\Big]=\EE\Big[\int^\infty_0{\lambda e^{-\lambda t}X_t\Phi(C_t)dt}\Big]
\end{align}
and
\begin{align}
\EE\Big[\int^\tau_0{X_t}d\nu_t\Big]=&\EE\Big[\int^\infty_0{\lambda e^{-\lambda s}\Big(\int^s_0{X_td\nu_t}\Big)ds}\Big]\nonumber\\
=&\EE\Big[\int^\infty_0{\Big(\int^\infty_t{\lambda e^{-\lambda s}ds}\Big)X_td\nu_t}\Big]=\EE\Big[\int^\infty_0{e^{-\lambda t}X_t d\nu_t}\Big],
\end{align}
where the integrals were exchanged by an application of Fubini's theorem. It then follows that problem \eqref{vfun00} may be rewritten as in \eqref{valuefunction} and \eqref{nonconvex}.

\section{Facts on the Ornstein-Uhlenbeck Process}
\label{factsOU}
\renewcommand{\theequation}{B-\arabic{equation}}

Recall the Ornstein-Uhlenbeck process $X$ of \eqref{OU}. It is well known that $X$ is a positively recurrent Gaussian process (cf., e.g., \cite{BorodinSalminen}, Appendix 1, Section 24, pp.\ 136-137) with state space $\mathbb{R}$ and that \eqref{OU} admits the explicit solution
\beq
\label{OUexplicit}
X^x_t= \mu + (x-\mu)e^{-\theta t} + \int_0^t \sigma e^{\theta(s-t)}dB_s.
\eeq
We introduced its infinitesimal generator $\mathbb{L}_{X}$ in \eqref{def:LX};
the characteristic equation $\mathbb{L}_{X}u = \lambda u$, $\lambda > 0$, admits the two linearly independent, positive solutions (cf.~\cite{JYC}, p.\ 280)
\beq
\label{phi}
\phi_{\lambda}(x):=
e^{\frac{\theta(x-\mu)^2}{2\sigma^2}}D_{-\frac{\lambda}{\theta}}\Big(\frac{(x-\mu)}{\sigma}\sqrt{2\theta}\Big)
\eeq
and
\beq
\label{psi}
\psi_{\lambda}(x):=
e^{\frac{\theta(x-\mu)^2}{2\sigma^2}}D_{-\frac{\lambda}{\theta}}\Big(-\frac{(x-\mu)}{\sigma}\sqrt{2\theta}\Big),
\eeq
which are strictly decreasing and strictly increasing, respectively. In both \eqref{phi} and \eqref{psi} $D_{\alpha}$ is the cylinder function of order $\alpha$ (see \cite{Trascendental}, Chapter VIII, among others) and it is also worth recalling that (see, e.g., \cite{Trascendental}, Chapter VIII, Section 8.3, eq.\ (3) at page 119)
\begin{align}
\label{cylinder}
D_{\alpha}(x):= \frac{e^{-\frac{x^2}{4}}}{\Gamma(-\alpha)}\int_0^{\infty}t^{-\alpha -1} e^{-\frac{t^2}{2} - x t} dt, \quad \text{Re}(\alpha)<0,
\end{align}
where $\Gamma(\cdot)$ is the Euler's Gamma function.

We denote by $\PP_x$ the probability measure on $(\Omega, \cF)$ induced by the process $(X^x_t)_{t\ge0}$, i.e.~such that $\mathbb{P}_x(\,\cdot\,) = \mathbb{P}(\,\cdot\,| X_0=x)$, $x \in \mathbb{R}$, and by $\EE_x[\,\cdot\,]$ the expectation under this measure.  Then, it is a well known result on one-dimensional regular diffusion processes (see, e.g., \cite{BorodinSalminen}, Chapter I, Section 10) that
\begin{equation}
\label{hittingtimes}
\EE_x[e^{-\lambda\tau_y}]=
\left\{
\begin{array}{ll}
\displaystyle \frac{\phi_{\lambda}(x)}{\phi_{\lambda}(y)}, \quad x \geq y,\\
\\
\displaystyle \frac{\psi_{\lambda}(x)}{\psi_{\lambda}(y)}, \quad x \leq y,
\end{array}
\right.
\end{equation}
with $\phi_{\lambda}$ and $\psi_{\lambda}$ as in \eqref{phi} and \eqref{psi} and $\tau_{y}:=\inf\{t \geq 0: X^x_t = y\}$ the hitting time of $X^x$ at level $y \in \mathbb{R}$. Due to the recurrence property of the Ornstein-Uhlenbeck process $X$ one has $\tau_{y} < \infty$ $\PP_x$-a.s.\ for any $x,y \in \R$.


\section{Some Proofs from Section \ref{Case2}}
\renewcommand{\theequation}{C-\arabic{equation}}
\label{app-proof}

\subsection{Proof of Theorem \ref{verifthm01}}
\label{app-proof1}

The proof goes through a number of steps which we organise in Lemmas, Propositions and Theorems. Integrating by parts in \eqref{opt-st} and noting that the martingale $(\int_0^{t}e^{-\lambda s}\sigma dB_s)_{t \geq 0}$ is uniformly integrable we can write
\bea
\label{osprob}
u(x;c) := v(x;c) + x =
\sup_{\sigma\ge0}\EE\bigg[\int_0^{\sigma}e^{-\lambda s}\left[k(c)X^x_{s} - \theta \mu \right]ds\bigg],
\eea
with $k(c)$ as in \eqref{def-k}. For each $c \in [0,1]$ we define the continuation and stopping regions of problem \eqref{osprob} by
\beq
\label{regions}
\mathcal{C}_c: = \{x: u(x;c)>0\} \quad \mbox{and} \quad \DD_c := \{x: u(x;c)=0\},
\eeq
respectively. From standard arguments based on exit times from small balls one notes that $\DD_c \subset \{x: x \leq \frac{\theta \mu}{k(c)}\}$ as it is never optimal to stop immediately in its complement $\{x: x > \frac{\theta \mu}{k(c)}\}$. Since $x \mapsto u(x;c)$ is increasing, $\DD_c$ lies below $\mathcal{C}_c$ and we also expect the optimal stopping strategy to be of threshold type.

Now, for any given $c\in[0,1]$ and ${\beta(c)}\in\mathbb{R}$ we define the hitting time $\sigma_\beta(x,c):=\inf\{t \geq 0:X^x_t \leq {\beta(c)}\}$. For simplicity we set $\sigma_\beta(x,c)=\sigma_\beta$. A natural candidate value function for problem \eqref{osprob} is of the form
\beq
\label{uequation}
u^\beta(x;c)= \left\{\begin{array}{ll} \displaystyle \EE\bigg[\int_0^{\sigma_\beta}e^{-\lambda s}\left(k(c)X^x_{s} - \theta \mu \right)ds\bigg], & x > {\beta(c)}, \\[+16pt]
 0, & x \leq {\beta(c)}.\end{array}
\right.
\eeq
An application of Fubini's theorem, \eqref{OUexplicit} and some simple algebra leads to
\begin{lemma}
\label{lemmaG}
For all $(x,c)\in\RR\times [0,1]$ and with $G$ as in \eqref{Gexpression} one has
\begin{eqnarray}
& \displaystyle \EE\bigg[\int_0^{\infty}e^{-\lambda s}\left(k(c)X^x_{s} - \theta \mu \right)ds\bigg]=G(x;c).
\end{eqnarray}
\end{lemma}
Recall $\LL$ and $\phi_\lambda$ as in the statement of Theorem \ref{verifthm01}. The analytical expression of $u^\beta$ is provided in the next
\begin{lemma}
\label{lemma1}
For $u^\beta$ as in \eqref{uequation} it holds
\begin{align}
u^\beta(x;c)=\left\{\begin{array}{ll}
\displaystyle G(x;c)-\frac{G({\beta(c)};c)}{\phi_{\lambda}({\beta(c)})}\phi_{\lambda}(x), & x > {\beta(c)} \\ [+10pt]
0, & x \leq {\beta(c)}.\end{array}
\label{def-ub}
\right.
\end{align}
\end{lemma}
\begin{proof}
From \eqref{uequation}, \eqref{Gexpression} and strong Markov property we have that for all $x>{\beta(c)}$
\bea
u^\beta(x;c)
 \hspace{-0.25cm} &=&  \hspace{-0.25cm} G(x;c) - \EE\bigg[\EE\bigg[\int^{\infty}_{\sigma_\beta}e^{-\lambda s}\left(k(c)X^x_{s} - \theta \mu \right)ds \Big|\mathcal{F}_{\sigma_\beta}\Big]\bigg] \\
 \hspace{-0.25cm} &=&  \hspace{-0.25cm} G(x;c) - \EE\Big[e^{-\lambda \sigma_\beta}G(X^x_{\sigma_\beta};c)\Big],\nonumber \\
 \hspace{-0.25cm} &=&  \hspace{-0.25cm} G(x;c) - G({\beta(c)};c) \frac{\phi_{\lambda}(x)}{\phi_{\lambda}({\beta(c)})}, \nonumber
\eea
where the last equality follows since $X^x$ is positively recurrent and by using well known properties of hitting times summarised in Appendix \ref{factsOU} for completeness (cf.~\eqref{hittingtimes}).

\end{proof}
The candidate optimal boundary ${\beta_*(c)}$ is found by imposing the familiar {\em principle of smooth fit}, i.e.\ the continuity of the first derivative $u^\beta_x$ at the boundary $\beta_*$. This amounts to solving problem \eqref{smfit01}.
\begin{prop}
\label{bstar01}
Recall \eqref{x0-01}. For each $c\in[0,1]$ there exists a unique solution ${\beta_*(c)}\in(-\infty, x_0(c))$ of \eqref{smfit01}.
Moreover, $\beta_*\in C^1([0,1])$ and it is strictly decreasing.
\end{prop}
\begin{proof}
Since we are only interested in finite valued solutions of \eqref{smfit01} and $\phi_{\lambda}(x)>0$ for all $x\in(-\infty,+\infty)$ we may as well consider the equivalent problem of finding $x\in\mathbb{R}$ such that $H(x;c)=0$, where
\begin{align}\label{def-H01}
H(x;c) := G_x(x;c)\phi_{\lambda}(x)-G(x;c)\phi'_{\lambda}(x).
\end{align}
We first notice that $G(x_0(c);c)=0$ (cf.~\eqref{Gexpression} and \eqref{x0-01}) and since $k(c)>0$, then \emph{(i)} $G(x;c)>0$ for $x>x_0(c)$, \emph{(ii)} $G(x;c)<0$ for $x<x_0(c)$ and \emph{(iii)} $G_x(x;c)>0$ for all $x$. Hence
\begin{align}\label{eq:Hx0}
H(x_0(c);c)=G_x(x_0(c);c)\phi_{\lambda}(x_0(c))>0.
\end{align}
Recall also that $\phi_{\lambda}$ is strictly convex (cf.\ \eqref{phi} and \eqref{cylinder} in Appendix \ref{factsOU}), then it easily follows by \eqref{Gexpression} and \eqref{def-H01} that
\begin{align}\label{smfit02}
H_x(x;c)= -G(x;c)\phi''_{\lambda}(x)>0,\qquad\text{for $x<x_0(c)$.}
\end{align}
Moreover, $H(x;c)>0$ for all $x \geq x_0(c)$ and so if ${\beta_*(c)}$  exists such that $H({\beta_*(c)};c)=0$ then ${\beta_*(c)}<x_0(c)$. Derivation of \eqref{smfit02} with respect to $x$ gives
\begin{equation*}
H_{xx}(x;c) = -G_x(x;c)\phi''_{\lambda}(x)-G(x;c)\phi'''_{\lambda}(x) <0,\qquad\text{for $x<x_0(c)$,}
\end{equation*}
which implies that $x \mapsto H(x;c)$ is continuous, strictly increasing and strictly concave on
$(-\infty, x_0(c))$. Hence, by \eqref{eq:Hx0} there exists a unique ${\beta_*(c)}<x_0(c)$ solving $H({\beta_*(c)};c)=0$ (and equivalently \eqref{smfit01}).
Since $H_x(\beta_*(c);c)>0$ for all $c\in[0,1]$ (cf.~\eqref{smfit02}), then $\beta_*\in C^1([0,1])$ from the implicit function's theorem, with
\begin{align}\label{eq:derivbeta}
\beta^\prime_*(c)=-\,\frac{H_c({\beta_*(c)};c)}{H_x({\beta_*(c)};c)},\quad c\in[0,1].
\end{align}

We now show that $c\mapsto {\beta_*(c)}$ is strictly decreasing. A direct study of the sign of the right-hand side of \eqref{eq:derivbeta} seems non-trivial so we use a different trick. It is not hard to verify from \eqref{x0-01} that $c \mapsto x_0(c)$ is strictly decreasing since $c \mapsto \Phi'(c)$ is strictly increasing. Setting $\bar{x}:={\beta_*(c)}$ in \eqref{smfit01}, straightforward calculations give
\begin{equation*}
\frac{\phi'_{\lambda}(\bar{x})}{\phi_{\lambda}(\bar{x})} = \frac{ G_x(\bar{x};c)}{G(\bar{x};c)} = \frac{ k(c)}{\bar{x} k(c) + \mu\theta\Phi'(c)} = \frac 1 {\bar{x}-x_0(c)}
\end{equation*}
so that $c \mapsto \frac{ G_x(\bar{x};c)}{G(\bar{x};c)}$ is strictly decreasing. Since $c\mapsto x_0(c)$ is continuous it is always possible to pick $c'>c$ sufficiently close to $c$ so that $\bar{x}< x_0(c')<x_0(c)$ (hence $G(\bar{x};c')<0$) and one finds
\begin{align}\label{monotbeta}
\frac{ G_x(\bar{x};c')}{G(\bar{x};c')} < \frac{\phi'_{\lambda}(\bar{x})}{\phi_{\lambda}(\bar{x})}
\end{align}
and therefore $H(\bar{x};c')>0$. It follows that $\beta_*(c')<{\beta_*(c)}$, since $x \mapsto H(x;c)$ is increasing for $x<x_0(c')$. Then $c \mapsto {\beta_*(c)}$ is a strictly decreasing map.
\end{proof}
We verify the optimality of $\beta_*$ in the next theorem and note that a stopping time $\sigma$ is optimal for \eqref{osprob} if and only if it is optimal for \eqref{opt-st}.
\begin{theorem}
\label{verifthm02}
The boundary $\beta_*$ of Proposition \ref{bstar01} is optimal for \eqref{osprob} in the sense that $\sigma^*$ of \eqref{eq:sigmastar}
is an optimal stopping time and $u^{\beta_*}\equiv u$.
\end{theorem}
\begin{proof}
The candidate value function $u^{\beta_*}$ (cf.\ \eqref{def-u00}) is such that $u^{\beta_*}(\cdot;c) \in C^1(\mathbb{R})$ by Proposition \ref{bstar01} and
it is convex. Hence it is also nonnegative, since $u_x^{\beta_*}({\beta_*(c)};c)=u^{\beta_*}({\beta_*(c)};c)=0$ by \eqref{def-u00} and \eqref{smfit01}.


It is easily checked that
\bea \label{eq:diffu}
(\mathbb{L}_X-\lambda)u^{\beta_*}(x;c)=\left\{\begin{array}{ll}\theta\mu-k(c)x, & x > {\beta_*(c)}, \\ [+8pt]
0, & x \le {\beta_*(c)}.
\end{array}
\right.
\eea
We claim (and we will prove it later) that
\beq
\label{eq:xhat}
{\beta_*(c)}<\frac{\theta \mu}{k(c)}=:\hat{x}_0(c)
\eeq
so that $(\mathbb{L}_X-\lambda)u^{\beta_*}(x;c) \leq \theta\mu-k(c)x$ for all $x \in \mathbb{R}$.

Fix $(x,c)\in\RR\times[0,1]$. Take now $R>0$ such that ${\beta_*(c)} \in (-R,R)$ and define $\tau_R:=\inf\{t \geq 0: X^x_t \notin (-R,R)\}$. By convexity of $u^{\beta_*}(\cdot,c)$, the It\^o-Tanaka formula (see, for example, \cite{KS}, Chapter 3, Section 3.6 D) and the principle of smooth fit we have
\begin{align}\label{eq:IT01}
\EE\left[ e^{-\lambda(\tau_R \wedge \tau)}u^{\beta_*}(X^x_{\tau_R \wedge \tau},c)\right] \leq u^{\beta_*}(x,c) + \EE\left[\int_0^{\tau_R \wedge \tau} e^{-\lambda s}\big(\theta \mu - k(c)X^x_s\big)ds \right],
\end{align}
for an arbitrary $\PP$-a.s.~finite stopping time $\tau\ge0$. 
Now $\tau_R \wedge \tau \uparrow \tau$ as $R \uparrow \infty$ and the integral inside the expectation on the right-hand side of \eqref{eq:IT01} is uniformly integrable. Then taking limits as $R\uparrow\infty$ and using that $u^{\beta_*}\ge0$ we obtain
\begin{equation*}
u^{\beta_*}(x;c) \geq
\EE\left[\int_0^{\tau} e^{-\lambda s}\big(\theta \mu - k(c)X^x_s\big)ds \right].
\end{equation*}
Since $\tau$ is arbitrary we can take the supremum over all stopping times to obtain $u^{\beta_*}\ge u$.

To prove the reverse inequality we take $\tau=\sigma^*$ to have strict inequality in \eqref{eq:IT01}. Then we notice that $0\le u^{\beta_*}(x,c)\le |G(\beta_*(c);c)|+|G(x;c)|$ for $x>\beta_*(c)$ so that recurrence of $X^x$ implies that
\begin{align}\label{eq:uiub}
\big(e^{-\lambda \tau}u^{\beta_*}(X_{\tau}^x,c)\big)_{\tau \geq 0}\:\text{is uniformly integrable and}\: e^{-\lambda \sigma^*}u^{\beta_*}(X^x_{\sigma^*};c) = e^{-\lambda\,\sigma^*}u^{\beta_*}(\beta_*(c),c).
\end{align}
Therefore
\begin{align}
\lim_{R\to\infty}\EE\left[ e^{-\lambda\,(\tau_R \wedge \sigma^*)}u^{\beta_*}(X^x_{\tau_R \wedge \sigma^*},c)\right]=\EE\left[ e^{-\lambda\sigma^*}u^{\beta_*}(\beta_*(c),c)\right]=0,
\end{align}
and in the limit we find $u^{\beta_*} = u$.

To conclude the proof we only need to show that \eqref{eq:xhat} holds true. Set $\hat{x}_0=\hat{x}_0(c)$ for simplicity. We have
\beq
\frac{H(\hat{x}_0;c)}{\phi_{\lambda}(\hat{x}_0)}= \frac{k(c)}{\lambda+\theta}-\frac{\theta\mu(k(c)-\theta)}{\lambda(\lambda+\theta)}\frac{\phi_{\lambda}'(\hat{x}_0)}{\phi_{\lambda}(\hat{x}_0)}
\label{eq:Hfrac}
\eeq
by \eqref{Gexpression}, \eqref{def-H01} and \eqref{x0-01}; since $\big(\mathbb{L}_X-\lambda\big)\phi_{\lambda} = 0$ and $\phi''_{\lambda} >0$ we also have
\begin{align}
\theta(\mu-\hat{x}_0)\phi'_{\lambda}(\hat{x}_0)-\lambda\phi_{\lambda}(\hat{x}_0)<0.
\label{eq:ode1}
\end{align}
It is clear that if $k(c)\ge\theta$ then the right-hand side of \eqref{eq:Hfrac} is strictly positive and ${\beta_*(c)}<\hat{x}_0(c)$. On the other hand, if $k(c)<\theta$ then $\mu-\hat{x}_0<0$ and from \eqref{eq:ode1} we get
\begin{align}\label{boundphi}
\frac{\phi'(\hat{x}_0)}{\phi(\hat{x}_0)}>\frac{\lambda}{\theta\mu}\left( \frac{k(c)}{k(c)-\theta}\right).
\end{align}
Now plugging \eqref{boundphi} into the right-hand side of \eqref{eq:Hfrac} we find $H(\hat{x}_0;c)/\phi_{\lambda}(\hat{x}_0) > 0$ so that again ${\beta_*(c)}<\hat{x}_0(c)$.
\end{proof}
\vspace{0.25cm}

\subsection{Proof of Proposition \ref{C-opt}}
\label{app-proof2}

By monotonicity of $g_*$ we have
\begin{equation*}
C^*_t = c+\nu_t^*=  c+\Big[g_*\big(\inf_{0 \leq s \leq t} X^x_s\big)-c\Big]^+ \geq g_*(X_t^x)\wedge 1=g_*(X_t^x),
\end{equation*}
since $0\le g_*\le 1$. Hence $1$ follows.

To prove $2$ fix $\omega \in \Omega$ and suppose that for some $t>0$ we have $(C^*_t(\omega), X^x_t(\omega)) \in \mathcal{C}$, i.e.~$C^*_t(\omega)>g_*(X^x_t(\omega))$. We distinguish two cases.
In the case that $g_*\left(\inf_{0 \leq u \leq t} X^x_u(\omega)\right) \geq c$, we have $g_*\left(\inf_{0 \leq u \leq t} X^x_u (\omega) \right)=C^*_t(\omega)>g_*(X_t^x(\omega))$ and then by monotonicity of $g_*$ we have $\inf_{0 \leq u \leq t} X^x_u(\omega)<X_t^x(\omega)$. By continuity of $t\mapsto X_t^x(\omega)$ we deduce that $r \mapsto \inf_{0 \leq u \leq r} X^x_u(\omega)$ is constant in the interval $r \in [t,t+\epsilon(\omega))$ for some $\epsilon(\omega)>0$.
In the case that $g_*\left(\inf_{0 \leq u \leq t} X^x_u(\omega)\right) < c$, we have $c=C^*_t(\omega)>g_*(X_t^x(\omega))$ and then again by monotonicity and continuity of $g_*$, and continuity of $X_t^x(\omega)$, there exists $\epsilon(\omega)>0$ such that
$c > g_*\left(\inf_{0 \leq u \leq t+\epsilon(\omega)} X^x_u(\omega)\right)$
and so $\nu_r^*(\omega)=0$ for all $r \in [0,t+\epsilon(\omega))$.

Summarising, we have shown that if $(C^*_t(\omega), X^x_t(\omega)) \in \mathcal{C}$ then $\nu^*$ is constant in a right (stochastic) neighbourhood of $t$, establishing the second part.


\bigskip

\textbf{Acknowledgments.} This work was started when the authors were visiting the Hausdorff Research Institute for Mathematics (HIM) at the University of Bonn in the framework of the Trimester Program ``Stochastic Dynamics in Economics and Finance''. We thank HIM for the hospitality. Part of this work was carried out during a visit of the second author at the School of Mathematics of the University of Manchester. The hospitality of this institution is gratefully aknowledged. We wish also to thank G.~Peskir, F.~Riedel and M.~Zervos for many useful discussions.


\end{document}